\documentclass[12pt]{amsart}
\usepackage{amsfonts, amsbsy, amsmath, amssymb}
\usepackage{xcolor}
\usepackage{tabularx}

\newtheorem{thm}{Theorem}[section]
\newtheorem{lem}[thm]{Lemma}
\newtheorem{cor}[thm]{Corollary}
\newtheorem{prop}[thm]{Proposition}
\newtheorem{exmp}[thm]{Example}

\newtheorem{conj}[thm]{Conjecture}

\newtheorem{rmk}[thm]{Remark}

\newtheorem{thm-con}[thm]{Theorem-Conjecture}
\numberwithin{equation}{section}

\theoremstyle{definition}

\title[Reversed Dickson polynomials]{Reversed Dickson polynomials}

\begin{document}

\author[Fang]{Jiaqi Fang} 
	\address{College of the Holy Cross, Worcester, MA, USA}
	\email{jfang24@g.holycross.edu}

\author[Fernando]{Neranga Fernando} 
	\address{Department of Mathematics and Computer Science,
College of the Holy Cross, Worcester, MA, USA}
	\email{nfernand@holycross.edu}
 
\author[Wu]{Haoming Wu} 
	\address{College of the Holy Cross, Worcester, MA, USA}
	\email{hwu23@g.holycross.edu} 

\thanks{The authors would like to thank the Weiss Summer Research Program at College of the Holy Cross for the support.}

\maketitle

\begin{abstract}
We investigate fixed points and cycle types of permutation polynomials and complete permutation polynomials arising from reversed Dickson polynomials of the first kind and second kind over $\mathbb{F}_p$. We also study the permutation behaviour of reversed Dickson polynomials of the first kind and second kind over $\mathbb{Z}_m$. Moreover, we prove two special cases of a conjecture on the permutation behaviour of reversed Dickson polynomials over $\mathbb{F}_p$. 
\end{abstract}

\section{Introduction}\label{S1}

Let $p$ be a prime, $\mathbb{F}_p$ be the finite field with $p$ elements and $\mathbb{F}_p^{\times}=\mathbb{F}_p\setminus \{0\}$. A polynomial $f\in \mathbb{F}_p[x]$ is called a \textit{permutation polynomial} of $\mathbb{F}_p$ if the associated mapping $x\mapsto f(x)$ from $\mathbb{F}_p$ to $\mathbb{F}_p$ is a permutation of $\mathbb{F}_p$. Every function from $\mathbb{F}_p$ to $\mathbb{F}_p$ can be represented by a unique polynomial in $\mathbb{F}_p[x]$. That is, if $\phi:\mathbb{F}_p\rightarrow \mathbb{F}_p$ is an arbitrary function, then there exists a unique polynomial $g\in \mathbb{F}_p[x]$ of degree at most $q-1$ such that $g(c)=\phi(c)$ for all $c\in \mathbb{F}_q$. The polynomial $g$ can be found by the Lagrange's interpolation for the function $\phi$. If $\phi$ is already given as a polynomial function, say $\phi: c\mapsto f(c)$ where $f\in \mathbb{F}_p[x]$, then $g$ can be obtained from $f$ by reduction modulo $x^p-x$. Finding new classes of permutation polynomials is an important topic in the area of finite fields as they play a central role in the algebraic and combinatorial aspects of finite fields. Permutation polynomials have many applications in the areas of Coding Theory, Cryptography, Combinatorics, Finite Geometry, and Computer Science, among other fields. Here are two examples of permutation polynomials of $\mathbb{F}_p$:

\begin{exmp}
Every linear polynomial $ax+b$, where $a\in \mathbb{F}_p^{\times}$ and $b\in \mathbb{F}_p$, is a permutation polynomial of $\mathbb{F}_p$.
\end{exmp}

\begin{exmp}
The monomial $x^n$ is a permutation polynomial of $\mathbb{F}_p$ if and only if $\text{gcd}(n,p-1)=1$. 
\end{exmp}

Charles Hermite studied the permutation behaviour of polynomials over finite prime fields. It was Leonard Eugene Dickson who studied polynomials for their permutation behaviour over finite arbitrary fields. The $n$-th Dickson polynomial is given by 

$$\displaystyle{D_n(x,a)=\sum_{i=0}^{\lfloor n/2\rfloor}\,\frac{n}{n-i}\binom{n-i}{i}\,(-a)^i}\,x^{n-2i},$$

where $a\in \mathbb{F}_p$ is a parameter. The recurrence relation of Dickson polynomials $D_n(x,a)$ is given by 
$$D_0(x,a)=2, \,\,D_1(x,a)=x, \,\,D_n(x,a)=xD_{n-1}(x,a)-aD_{n-2}(x,a),\,\,\text{for}\,n\geq 2.$$

L. E. Dickson studied these polynomials for their permutation property over finite fields. The permutation property of Dickson polynomials is completely known. When $a=0$, we have $D_n(x,0)=x^n$, which is a permutation polynomial of $\mathbb{F}_p$ if and only if $\text{gcd}(n,p-1)=1$. In \cite{NB68}, N\"{o}bauer proved that $D_n(x,a)$, where $a\neq 0$,  is a permutation polynomial of $\mathbb{F}_p$ if and only if $\text{gcd}(n,p^2-1)=1$. 

In 1923, Issai Schur named the polynomials $D_n(x,a)$ in Dickson's honor. He also introduced a variant of Dickson polynomials which are called Dickson polynomials of the second kind \cite{Shur23}. Because of this reason, the polynomials given by $D_n(x,a)$ are now called Dickson polynomials of the first kind.

The $n$-th Dickson polynomial of the second kind is given by 

$$\displaystyle{E_n(x,a)=\sum_{i=0}^{\lfloor n/2\rfloor}\,\binom{n-i}{i}\,(-a)^i}\,x^{n-2i},$$

where $a\in \mathbb{F}_p$ is a parameter. The recurrence relation of Dickson polynomials of the second kind $E_n(x,a)$ is given by 
$$E_0(x,a)=1, \,\,E_1(x,a)=x, \,\,E_n(x,a)=xE_{n-1}(x,a)-aE_{n-2}(x,a),\,\,\text{for}\,n\geq 2.$$

The permutation property of Dickson polynomials of the second kind is not completely known yet. We refer the reader to the monograph \cite{DM} for more details on Dickson polynomials, and to \cite{CC} for the permutation behaviour of reversed Dickson polynomials of the second kind. 

Reversed Dickson polynomials of the first kind were introduced in 2009 by Xiang-dong Hou, Gary L. Mulen, Joseph L. Yucas and James Sellers by reversing the roles of the variable and the parameter in Dickson polynomials of the first kind \cite{HMSY}. The $n$th reversed Dickson polynomial of the first kind is given by the explicit expression 

$$\displaystyle{D_n(a,x)=\sum_{i=0}^{\lfloor n/2\rfloor}\,\frac{n}{n-i}\,\binom{n-i}{i}\,a^{n-2i}\,(-x)^i},$$

where $a\in \mathbb{F}_p$ is a parameter. The recurrence relation of reversed Dickson polynomials is given by 

$$D_0(a,x)=2, \,\,D_1(a,x)=a,\,\,\,D_n(a,x)=aD_{n-1}(a,x)-xD_{n-2}(a,x)\,\,\,\text{for}\,\,n\geq 2.$$

In \cite{HMSY}, the authors studied these polynomials for their permutation behaviour over finite fields. Their work showed that the family of reversed Dickson polynomials, like its counterpart $D_n(x,a)$, is also a rich source of permutation polynomials. They showed that reversed Dickson polynomials are closely related to \textit{almost perfect nonlinear} (APN) functions which have many applications in Cryptography. 

Since 2009, many have studied reversed Dickson polynomials of the first kind for their permutation property over finite fields. In \cite{HL}, Xiang-dong Hou and Tue Ly found necessary conditions for reversed Dickson polynomials of the first kind to be permutational. 

In a similar way, one can define reversed Dickson polynomials of the second kind by interchanging the variable and the parameter in Dickson polynomials of the second kind. The $n$th reversed Dickson polynomial of the second kind is defined by 

$$\displaystyle{E_n(a,x)=\sum_{i=0}^{\lfloor n/2\rfloor}\,\binom{n-i}{i}\,a^{n-2i}\,(-x)^i},$$

where $a\in \mathbb{F}_p$ is a parameter. In \cite{HQZ}, Shaofang Hong, Xiaoer Qin and Wei Zhao found several necessary conditions for reversed Dickson polynomials of the second kind to be permutational.

In 2012, Qiang (Steven) Wang and Joseph L. Yucas introduced Dickson polynomials of the $(k+1)$-th kind \cite{WY}. For $a\in \mathbb{F}_p$, the $n$-th Dickson polynomial of the $(k+1)$-th kind $D_{n,k}(x,a)$ is defined by 
$$\displaystyle{D_{n,k}(x,a)=\sum_{i=0}^{\lfloor n/2\rfloor}\,\frac{n-ki}{n-i}\binom{n-i}{i}\,(-a)^i}\,x^{n-2i},$$
and $D_{0,k}(x,a)=2-k$. Note that $D_{n,0}(x,a)=D_n(x,a)$ and $D_{n,1}(x,a)=E_n(x,a)$. In \cite{WY}, the authors studied general properties of Dickson polynomials of the $(k+1)$-th kind, and in particular, they studied the permutation behaviour of Dickson polynomials of the third kind, $D_{n,2}(x,a)$, over finite fields. They found some necessary conditions for the Dickson polynomials of the third kind to be permutational over finite fields. Moreover, they completely explained the permutation property of Dickson polynomials of the third kind over finite prime fields. We refer the reader to \cite{WY} for more details on reversed Dickson polynomials of the $(k+1)$-th kind. 

In the same paper, Wang and Yucas introduced reversed Dickson polynomials of the $(k+1)$-th kind. For $a\in \mathbb{F}_p$, the $n$-th reversed Dickson polynomial of the $(k+1)$-th kind $D_{n,k}(a,x)$ is defined by 
$$\displaystyle{D_{n,k}(a,x)=\sum_{i=0}^{\lfloor n/2\rfloor}\,\frac{n-ki}{n-i}\binom{n-i}{i}\,a^{n-2i}\,(-x)^i},$$
and $D_{0,k}(a,x)=2-k$. Note that $D_{n,0}(a,x)=D_n(a,x)$ and $D_{n,1}(a,x)=E_n(a,x)$. In \cite{NF}, \cite{NF1} and \cite{NF2}, the second author of this paper studied the general properties and permutation behaviour of reversed Dickson polynomials of the $(k+1)$-th kind over finite fields. 

The set of permutation polynomials of $\mathbb{F}_p$ of degree less than or equal to $p-1$ form a group under composition and reduction modulo $x^p-x$, which is symmetric to $S_p$. Ay\c{c}a \c{C}e\c{s}melio\u{g}lu, Wilfried Meidl, Alev Topuzo\u{g}lu studied the cycle structure of permutation polynomials over finite fields in \cite{CMT}, Shair Ahmad explained the cycle structure of Dickson permutation polynomials $D_n(x,a)$ for $a=0$ in \cite{Ahmad}, and Rudolf Lidl and Gary L. Mullen determined the cycle structure of Dickson permutation polynomials $D_n(x,a)$ for $a\in \{-1,1\}$. Moreover, Adama Diene and Mohamed D. Salim determined the number of fixed points of Dickson polynomials of the second kind $E_n(x,a)$. Motivated by their results, we determine the number of fixed points and the cycle type of reversed Dickson permutation polynomials of the first kind and second kind. 

Most studies on permutation polynomials have been over finite fields. However, there have been a few studies on permutation polynomials over finite rings $\mathbb{Z}_m$, where $m>1$ is an integer, due to their applications in coding theory and cryptography. We refer the reader to \cite{GHM}, \cite{Rivest}, \cite{ST}, \cite{T} for further details on permutation polynomials over $\mathbb{Z}_m$. In \cite{QD}, Longjiang Qu and Cunsheng Ding investigated the permutation property of the Dickson polynomials of the second kind $E_n(x,a)$ over $\mathbb{Z}_m$. Motivated by these results, we investigate the permutation behaviour of reversed Dickson polynomials of the first kind and second kind over finite rings. 

A permutation polynomial of $\mathbb{F}_p$ is called a complete permutation polynomial (complete mapping) of $\mathbb{F}_p$ if $f(x)+x$ is also a permutation polynomial of $\mathbb{F}_p$. Complete permutation polynomials have many applications in orthogonal latin squares, combinatorics, and cryptography. Harald Niederreiter and Karl H. Robinson conducted a detailed study of complete permutation polynomials over finite fields in \cite{NR82}. Gary L. Mullen and Harald Niederreiter determined the complete permutation polynomials arising from Dickson polynomials in \cite{MN87}. Motivated by their work, we study complete permutation polynomials from reversed Dickson polynomials of the first kind and second kind. We refer the reader to \cite{BZ}, \cite{MP14} and \cite{NR82} for more about complete permutation polynomials and their applications in cryptography. 

The paper is organized as follows. In Section~\ref{S2}, we give some preliminary details. In Section~\ref{S3}, we first present several known results on reversed Dickson polynomials of the first kind including a conjecture on their permutation behaviour over $\mathbb{F}_p$. We then explain the number of fixed points of each reversed Dickson permutation polynomial. We also explain the cycle type of the reversed Dickson permutation polynomials. Moreover, in Section~\ref{S3}, we discuss the permutation behaviour of reversed Dickson polynomials of the first kind over $\mathbb{Z}_{2^t}$, $\mathbb{Z}_{3}$, and $\mathbb{Z}_{3^t}$, and also explain when the $n$th reversed Dickson polynomials is a complete permutation polynomial over $\mathbb{Z}_3$. We then find sufficient conditions for the reversed Dickson polynomials of the first kind to be a complete permutation polynomials over $\mathbb{F}_p$, where $p\geq 5$, and end the section with a conjecture that classifies all complete permutation polynomials arising from reversed Dickson polynomials of the first kind. 

In Section~\ref{S4}, we explore the permutation behaviour of the reversed Dickson polynomials of the second kind. We begin the section with some preliminary details, and then explain necessary and sufficient conditions for reversed Dickson polynomials of the second kind to be a permutation polynomial when $p=2$ and $p=3$. This leads us to explore their permutation behaviour over finite rings $\mathbb{Z}_{2^t}$ and $\mathbb{Z}_{3^t}$. Finding zeros of the regular derivative of polynomials in $\mathbb{F}_p$ is an important task in the study of their permutation behaviour over finite rings. We first explain when $E_{n_1}(1,x)=E_{n_2}(1,x)$ for all $x\in \mathbb{F}_p$, and then $E_{n_1}^{\prime}(1,x)=E_{n_2}^{\prime}(1,x)$ for all $x\in \mathbb{F}_p\setminus \{\frac{1}{4}\}$. We note to the reader that the periodicity of the sequence $\{E_n^{\prime}\Big(1,\frac{1}{4}\Big)\}$ depends on whether $p$ is a Mersenne prime or not. Our studies on the sequence $\{E_n^{\prime}\Big(1,\frac{1}{4}\Big)\}$ lead to three Conjectures. We completely classify the permutation behaviour of $E_n(1,x)$ over $\mathbb{Z}_{2^t}$ and $\mathbb{Z}_{3^t}$. We then find sufficient conditions on $n$ for $E_n(1,x)$ to be a permutation polynomial and a complete permutation polynomials over $\mathbb{F}_p$, where $p>3$. Our numerical results strongly suggest that these values of $n$ are also necessary. Therefore, we end the section with four conjectures. 

In Section~\ref{S5} and Section~\ref{S6}, we prove two special cases of a conjecture on the permutation behaviour of reversed Dickson polynomials of the first kind. 

In Appendix~\ref{AppA}, we present the cycle types of permutation polynomials and complete permutation polynomials from reversed Dickson polynomials of the first kind and second kind. In Appendix~\ref{AppB}, we present computational results that are required in some proofs. In Appendix~\ref{AppC}, we present the Python codes to generate reversed Dickson permutation polynomials and reversed Dickson complete permutation polynomials. Moreover, we give Python codes to compute cycle types of reversed Dickson permutation polynomials and fixed points. 

Hereafter, we call permutation polynomials, complete permutation polynomials and reversed Dickson polynomials PPs, CPPs and RDPs, respectively, throughout the paper. 


\section{Preliminary details}\label{S2}

When $a\in \mathbb{F}_p^*$, it turns out that we only need to consider the case $a=1$ to study the permutation behaviour of RDPs over finite fields because

$$\displaystyle{D_{n,k}(a,x)=a^n\,\,\sum_{i=0}^{\lfloor n/2\rfloor}\,\frac{n-ki}{n-i}\,\binom{n-i}{i}\,\,(-1)^i\,\,\Big(\frac{x}{a^2}\Big)^i=a^n\,D_{n,k}\Big(1,\frac{x}{a^2}\Big)}.$$

We note to the reader that this is not true over finite rings in general. In this paper, we always assume that $a=1$ unless specified otherwise. 

As we discuss the permutation behaviour of reversed Dickson polynomials of the first kind and second kind over finite rings in latter sections, we give the following two lemmas that play an important role in determining the permutation polynomials over finite rings. 

\begin{lem}\label{L1}
If $m=ab$, where $\textnormal{gcd}(a,b)=1$, then $g(x)$ is a PP mod $m$ if and only if $g(x)$ is a PP mod $a$ and mod $b$.
\end{lem}

Because of the Fundamental Theorem of Arithmetic, we only need to explore PPs over $\mathbb{Z}_{p_i^{t_i}}$.

\begin{lem}\label{L2}
A polynomial $g(x)$ is a PP mod $p^t$, $t>1$, if and only if it is a PP mod $p$ and $g^{\prime}(s)\not\equiv 0\pmod{p}$ for every integer $s$, where $g^{\prime}(x)$ denotes the derivative of $g(x)$. 
\end{lem}

Let $\text{gcd}(a,m)$ denote the greatest common divisor of the integers $a$ and $m$. If $\text{gcd}(a,m)=1$, then $a$ is called a quadratic residue modulo $m$ if the congruence $x^2\equiv a\pmod{m}$ has a solution, otherwise $a$ is called a quadratic nonresidue modulo $m$. Let $p$ denote an odd prime. Then the Legendre symbol $\left(\frac{a}{p}\right)$ is defined by 

\[
\left(\frac{a}{p}\right) = 
\begin{cases}
 1 & \text{if}\,\,a\,\,\text{is a quadratic residue modulo}\,\,p,\\  
 -1 & \text{if}\,\,a\,\,\text{is a quadratic nonresidue modulo}\,\,p,\\
 0 &\text{if}\,\,a\,\,\text{is a multiple of}\,\,p.
\end{cases}
\]

For later use, we also note that for an odd prime $p$, we have $\displaystyle\left(\dfrac{a}{p}\right) \displaystyle\equiv a^\frac{p-1}{2},$

\[
\left(\frac{-1}{p}\right) = 
\begin{cases}
 1 & \text{if}\,\,p \equiv 1 \pmod{4},\\  
 -1 & \text{if}\,\,p \equiv 3 \pmod{4},\\
\end{cases}
\]

and 

\[
\left(\frac{3}{p}\right) = 
\begin{cases}
 1 & \text{if}\,\,p\,\,\equiv 1\,\,\text{or}\,\,11 \pmod{12},\\  
 -1 & \text{if}\,\,p\,\,\equiv 5\,\,\text{or}\,\,7 \pmod{12}.\\
\end{cases}
\]


\section{Reversed Dickson polynomials of the first kind}\label{S3}


\subsection{Background}

We first list some nice properties satisfied by RDPs over $\mathbb{F}_{p}$ that appeared in \cite{HMSY}. 

\begin{lem}\label{LJ1}

Let $p$ be a prime and $e$ a positive integer. Let $n \geq 0$ be an integer. 

\begin{enumerate} 
\item [(i)] $D_{np}\left(1, x\right)=(D_{n}(1, x))^p$ in $\mathbb{F}_{p}[x]$. 
\item [(ii)] If $n_1,n_2>0$ are integers such that $n_1\equiv n_2\pmod{p^2-1}$, then $D_{n_1}(1,x)=D_{n_2}(1,x)$ for all $x\in \mathbb{F}_p$. 
\item [(iii)] Let $n_{1}$ and $n_{2}$ be two positive integers. Then $D_{n_1}(1,x)$ is a PP on $\mathbb{F}_{p}$ if and only if $D_{n_2}(1,x)$ is a PP on $\mathbb{F}_{p}$ whenever $n_{1}$ and $n_{2}$ belong to the same $p$-cyclotomic coset modulo $p^2-1$.
\end{enumerate} 

\end{lem} 

The following conjecture that gives necessary and sufficient conditions for RDPs to be PPs appeared in \cite{HMSY}. The authors proved that the indices $n$ given in the conjecture are sufficient for RDPs to be PPs over $\mathbb{F}_p$. However, showing that they are also necessary remains an arduous task. In Theorem~\ref{T5} and Theorem~\ref{T7}, we prove two special cases of this conjecture: $p=5$ and $p = 7$. 
 
\begin{conj}\label{CCC1}
Let $p > 3$ be a prime and let $1\leq n \leq p^2-1.$ Then $D_n(1,x)$ is a PP on $\mathbb{F}_p$ if and only if 
\[
n = \left\{
\begin{alignedat}{2}
  & 2,2p,3,3p,p+1,p+2,2p+1 & \quad & \text{if}\ p\equiv 1 \pmod{12},\\
  & 2,2p,3,3p,p+1          & \quad & \text{if}\ p\equiv 5 \pmod{12},\\
  & 2,2p,3,3p,p+2,2p+1     & \quad & \text{if}\ p\equiv 7 \pmod{12}, \\
  & 2,2p,3,3p              & \quad & \text{if}\ p\equiv 11 \pmod{12};
\end{alignedat}
\right.
\]
\end{conj}


\subsection{Fixed points and cycle types of reversed Dickson permutation polynomials}

In this subsection, we investigate fixed points and cycle types of reversed Dickson permutation polynomials. We begin the subsection with a lemma of which proof is obvious. 

\begin{lem}\label{LJ61}
Every linear polynomial $ax+1$, where $a \in\mathbb{F}_p\setminus \{0,1\}$, has exactly 1 fixed point.
\end{lem}

\begin{prop}\label{PJ61}
Let $p\geq 3$ be an odd prime and $n\in \{2, 2p\}$. Then, the reversed Dickson permutation polynomial $D_n(1,x)$ has exactly one fixed point. 
\end{prop} 

\begin{proof}
The proof follows form the fact that each reversed Dickson permutation polynomial $D_n(1,x)$ in this case is a linear polynomial of the form $ax+1$, where $a \in\mathbb{F}_p\setminus \{0,1\}$. 
\end{proof}

\begin{prop}\label{PPJ61}
Let $p>3$ be an odd prime and $n\in \{3, 3p\}$. Then, the reversed Dickson permutation polynomial $D_n(1,x)$ has exactly one fixed point. 
\end{prop} 

\begin{proof}
The proof follows form the fact that each reversed Dickson permutation polynomial $D_n(1,x)$ in this case is a linear polynomial of the form $ax+1$, where $a \in\mathbb{F}_p\setminus \{0,1\}$. 
\end{proof}

\begin{thm}
Let $p\equiv 5\pmod{12}$. The permutation polynomial $D_{p+1}(1,x)$ has no fixed point. 
\end{thm} 

\begin{proof}
Consider the polynomial $D_{p+1}(1,x)$. From \cite[Proposition 5.1]{HMSY}, we have $$D_{p+1}(1,x)=\frac{1}{2}+\frac{1}{2}(1-4x)^{\frac{p+1}{2}}.$$
Assume that $D_{p+1}(1,x)$ has a fixed point, say $a$. Clearly $a\neq0$ since $D_{p+1}(1,0)=1.$ Since $a$ is a fixed point, we have $$a=\frac{1}{2}+\frac{1}{2}(1-4a)^{\frac{p+1}{2}},$$
which implies
\begin{equation}\label{EE1}
2a-1=(1-4a)^{\frac{p+1}{2}}.
\end{equation} 

Let $b=1-4a$. Since $a\neq 0$, $b\neq 1$. Also, note that $b\neq 0$ as $\frac{1}{4}$ is not a fixed point. From \eqref{EE1}, we have
$$\displaystyle{-\,\frac{(b+1)}{2}=b^{\frac{p+1}{2}}}.$$

Square both sides to obtain 
$$3b^2-2b-1=0,$$
which implies
$$(b-1)(3b+1)=0.$$

Since, $b\neq 1$, we have $b=-\frac{1}{3}$, that is, $a=\frac{1}{3}$. Since $a=\frac{1}{3}$ is a fixed point by assumption, we have

$$\frac{1}{3}=D_{p+1}\Big(1,\frac{1}{3}\Big)=\frac{1}{2}+\frac{1}{2}\Big(1-\frac{4}{3}\Big)^{\frac{p+1}{2}},$$

which can be simplified to $\Big(\frac{-1}{3}\Big)^{\frac{p+1}{2}}=1.$ This implies $\frac{-1}{3}$ is a quadratic residue in $\mathbb{F}_p$. This is a contradiction because $-1$ is a quadratic residue and $3$ is a quadratic nonresidue whenever $p\equiv 5\pmod{12}$, and thus $-\frac{1}{3}$ must be a quadratic nonresidue. This completes the proof. 
\end{proof} 

\begin{thm}\label{TTJ6}
Let $p\equiv 1\pmod{12}$. Then the permutation polynomial $D_{p+1}(1,x)$ has exactly one fixed point, and the permutation polynomials $D_{p+2}(1,x)$ and $D_{2p+1}(1,x)$ have exactly $\frac{p+1}{2}$ fixed points. 
\end{thm} 

\begin{proof}
We first consider the permutation polynomial $D_{p+1}(1,x)$. From \cite[Proposition 5.1]{HMSY}, we have $$D_{p+1}(1,x)=\frac{1}{2}+\frac{1}{2}(1-4x)^{\frac{p+1}{2}}.$$ 
Assume $a,b$ are fixed points of $D_{p+1}(1,x)$ where $a,b\neq0$ since $D_{p+1}(1,0)=1.$  Then, we have 
\begin{equation}\label{EJ61}
a=D_{p+1}(1,a)=\frac{1}{2}+\frac{1}{2}(1-4a)^{\frac{p+1}{2}}
\end{equation} 
and 
\begin{equation}\label{EJ62}
b=D_{p+1}(1,b)=\frac{1}{2}+\frac{1}{2}(1-4b)^{\frac{p+1}{2}}.
\end{equation}

From \eqref{EJ61} and \eqref{EJ62}, we have 
$$(1-4a)^{\frac{p+1}{2}}=2a-1$$
and
$$(1-4b)^{\frac{p+1}{2}}=2b-1,$$
respectively. Divide one by the other to obtain 
$$\Big(\frac{1-4a}{1-4b}\Big)^{\frac{p+1}{2}}=\frac{2a-1}{2b-1},$$
which implies
\begin{equation}\label{EJ63}
\Big(\frac{1-4a}{1-4b}\Big)^{\frac{p-1}{2}}\cdot \Big(\frac{1-4a}{1-4b}\Big)=\frac{2a-1}{2b-1}.
\end{equation}

Note that $a,b\neq \frac{1}{2}$ because $$D_{p+1}\Big(1,\frac{1}{2}\Big)=\frac{1}{2}+\frac{1}{2}(1-2)^{\frac{p+1}{2}}=\frac{1}{2}+\frac{1}{2}(-1)^{\frac{p+1}{2}}\neq \frac{1}{2},$$
and $a,b\neq \frac{1}{4}$ because
$$D_{p+1}\Big(1,\frac{1}{4}\Big)=\frac{1}{2}+\frac{1}{2}(1-1)^{\frac{p+1}{2}}=\frac{1}{2}\neq\frac{1}{4}.$$
If $\frac{1-4a}{1-4b}$ is a quadratic residue, then from \eqref{EJ63} we have $\frac{1-4a}{1-4b}=\frac{2a-1}{2b-1},$ which implies $a=b$. If $\frac{1-4a}{1-4b}$ is a quadratic nonresidue, then from \eqref{EJ63} we have $\frac{1-4a}{1-4b}=\frac{1-2a}{1-2b},$ which implies $a=b$. Thus the permutation polynomial $D_{p+1}(1,x)$ has exactly one fixed point. 

Let $c$ be a fixed point of $D_{p+2}(1,x)$. We show that there are $\frac{p+1}{2}$ choices for $c$. From \cite[Proposition 5.1]{HMSY}, we have 
$$D_{p+2}(1,x)= \frac{1}{2}+\frac{1}{2}(1-4x)^\frac{p+1}{2}-x.$$
Clearly, $c\neq0$ since $D_{p+2}(1,0)\neq0$. Since $c$ is a fixed point, we have $$c=D_{p+2}(1,c)=\frac{1}{2}(1-4c)^{\frac{p+1}{2}}+\frac{1}{2}-c,$$
which implies 
\begin{equation}\label{EJ64}
4c-1=(1-4c)^{\frac{p+1}{2}}.
\end{equation}

Let $d=1-4c$. Then \eqref{EJ64} becomes 
\begin{equation}\label{EJ65}
d\,(d^{\frac{p-1}{2}}+1)=0.
\end{equation}

Note that $d=0$, i.e. $c=\frac{1}{4}$, is a fixed of the polynomial $D_{p+2}(1,x)$. We show that there are $\frac{p-1}{2}$ more choices for $c$. If $d\neq 0$, from \eqref{EJ65}, we have
$$d^{\frac{p-1}{2}}=-1,$$
and this means $d$ is a quadratic nonresidue in $\mathbb{F}_p$. Therefore, there are $\frac{p-1}{2}$ more choices for $d$, that is, there are $\frac{p-1}{2}$ more choices for $c$. The rest of the proof follows from the fact that there are altogether $\frac{p+1}{2}$ choices for $c$, and $D_{p+2}(1,x)=D_{2p+1}(1,x)$ for all $x\in \mathbb{F}_p$. 
\end{proof} 

\begin{thm}\label{TTTTJ6}
Let $p\equiv 7\pmod{12}$. Then the reversed Dickson permutation polynomials $D_{p+2}(1,x)$ and $D_{2p+1}(1,x)$ have exactly $\frac{p+1}{2}$ fixed points. 
\end{thm} 

\begin{proof}
The proof is similar to that of Theorem~\ref{TTJ6}. 
\end{proof} 

\begin{thm}\label{TT1}
Let $p>3$ be a prime and $j\in\mathbb{Z}^+$ such that $j\ |\ p-1.$ Then, the cycle type of the permutation polynomials $D_2(1,x)$ and $D_{2p}(1,x)$ is $(\underbrace{\frac{p-1}{j},...,\frac{p-1}{j}}_{j\rm\ times},1),$ where $\text{ord}_p(-2)=\frac{p-1}{j}.$ In particular, if $-2$ is a primitive root modulo $p$, i.e. $j=1,$ then the cycle type of the permutation polynomials $D_2(1,x)$ and $D_{2p}(1,x)$ is $(p-1,1).$
\end{thm} 

\begin{proof}
Let $p>3$ be prime and $j\in\mathbb{Z}^+$ such that $j\vert p-1.$ Assume that $\text{ord}_p(-2)=\frac{p-1}{j}$. Since $D_{2}(1,x)=D_{2p}(1,x)$ for all $x\in \mathbb{F}_p$, we have $$D_2(1,x)=D_{2p}(1,x)=1+(p-2)x,$$
which has exactly one fixed point by Proposition~\ref{PJ61}, and the fixed point is $\frac{1}{3}$. 
Let $f(x)=1+(p-2)x$, and $f^{(n)}$ be the $n$th iterate of the polynomial $f(x)$. Let $a\in \mathbb{F}_p\setminus \{\frac{1}{3}\}$. Then we have 
$$f^{(n)}(a)=1+(p-2)+(p-2)^2+...+(p-2)^{n-1}+(p-2)^n\,a,$$
which implies
$$f^{(n)}(a)=\frac{1-(-2)^n}{1-(-2)}+(-2)^n\,a=\frac{1-(-2)^n}{3}+(-2)^n\,a.$$
Since $\text{ord}_p(-2)=\frac{p-1}{j}$, $f^{(\frac{p-1}{j})}(a)=a$, and $f^{(k)}\neq a$ for any $k<\frac{p-1}{j}$, where $k\,\vert \,p-1.$

We have shown that if there is a cycle of length greater than 1, then it must be of length $\frac{p-1}{j}$. From Proposition~\ref{PJ61}, we know that $f(x)$ has exactly one fixed point, and the fixed point is $\frac{1}{3}$. Therefore, the cycle type of the permutation polynomials $D_2(1,x)$ and $D_{2p}(1,x)$ is $(\underbrace{\frac{p-1}{j},...,\frac{p-1}{j}}_{j\rm\ times},1).$

In particular, if $-2$ is a primitive root modulo $p$, then $j=1.$
\end{proof}

\begin{thm}\label{TJ7T1}
Let $p=3.$ Then the cycle type of the permutation polynomials $D_{2}(1,x)$ and $D_{2p}(1,x)$ is $(3)$.
\end{thm} 

\begin{proof}
Let $p=3$, Since $D_{2}(1,x)=D_{2p}(1,x)$ for all $x\in \mathbb{F}_p$, we have $$D_2(1,x)=D_{2p}(1,x)=1+(p-2)x=1+x,$$
which has no fixed points. The proof follows from the fact that $(f\circ f)(x)\neq x$ for all $x\in \mathbb{F}_3$. 
\end{proof}

\begin{thm}
Let $p>3$ be a prime and $j\in\mathbb{Z}^+$ such that $j\,\vert \,p-1$. Then the cycle type of the permutation polynomials $D_3(1,x)$ and $D_{3p}(1,x)$ is $(\underbrace{\frac{p-1}{j},...,\frac{p-1}{j}}_{j\rm\ times},1)$ where $\text{ord}_p(-3)=\frac{p-1}{j}.$ In particular, if $-3$ is a primitive root modulo $p,$ then the cycle type of the permutation polynomials $D_3(1,x)$ and $D_{3p}(1,x)$ is $(p-1,1).$
\end{thm}

\begin{proof}
Let $p>3$ be prime and $j\in\mathbb{Z}^+$ such that $j\,\vert \,p-1.$ Assume that $\text{ord}_p(-3)=\frac{p-1}{j}$. Since $D_{3}(1,x)=D_{3p}(1,x)$ for all $x\in \mathbb{F}_p$, we $$D_3(1,x)=D_{3p}(1,x)=1+(p-3)x,$$ 
which has exactly one fixed point by Proposition~\ref{PPJ61}, and the fixed point is $\frac{1}{4}$. 
Let $f(x)=1+(p-3)x$, and $f^{(n)}$ be the $n$th iterate of the polynomial $f(x)$. Let $a\in \mathbb{F}_p\setminus \{\frac{1}{4}\}$. Then we have  
$$f^{(n)}(a)=1+(p-3)+(p-3)^2+...+(p-3)^{n-1}+(p-3)^n\,a,$$
which leads to 
$$f^{(n)}(a)=\frac{1-(-3)^n}{1-(-3)}+(-3)^n\,a=\frac{1-(-3)^n}{3}+(-3)^n\,a.$$
The rest of the proof is similar to that of Theorem~\ref{TT1}. 
\end{proof}

\begin{thm}\label{TJ7T2}
Let $p \geq 3$. Then the cycle type of the polynomial $D_2(1,x)+x$ is $(\underbrace{2,...,2}_{\frac{p-1}{2}\rm times},1)$.
\end{thm}

\begin{proof}
Let $p \geq 3$. Since $D_2(1,x)+x=1+(p-1)x$, it has exactly one fixed point by Lemma~\ref{LJ61}, and the fixed point is $\frac{1}{2}$. The proof follows from the fact that $(1-x)\circ (1-x)=x$ for all $\mathbb{F}_p\setminus \{\frac{1}{2}\}$. 
\end{proof}

\begin{thm}
Let $p \equiv 1 \pmod{12}$ or $p \equiv 7 \pmod{12}$.\\
Let $j \in \mathbb{Z}^+$ such that $j\vert p-1$. Then the permutation polynomials $D_{p+2}(1,x)$ and $D_{2p+1}(1,x)$ has the cycle type $(\underbrace{\frac{p-1}{j},...,\frac{p-1}{j}}_{\frac{j}{2}\rm\ times},\underbrace{1,...,1}_{\frac{p+1}{2}\rm times})$ whenever ord$_p(-3) = \frac{p-1}{j}$.
\end{thm}

\begin{proof}
Note that $$D_{p+2}(1,x)=D_{2p+1}(1,x)=\frac{1}{2}(1 - 4x)^\frac{p+1}{2} + \frac{1}{2} - x$$ for all $x\in \mathbb{F}_p$. By Theorem~\ref{TTJ6} and Theorem~\ref{TTTTJ6}, we know that $D_{p+2}(1,x)$ has $\frac{p+1}{2}$ fixed points. 

Let $a$ be a fixed point. Then 

$$a=D_{p+2}(1,a)=\frac{1}{2}(1 - 4a)^\frac{p+1}{2} + \frac{1}{2} - a.$$ 

We have

$$a=\frac{1}{2}(1 - 4a)^\frac{p+1}{2} + \frac{1}{2} - a \iff (1-4a)\,((1-4a)^{\frac{p-1}{2}}+1)=0,$$

which implies

$$a=\frac{1}{2}(1 - 4a)^\frac{p+1}{2} + \frac{1}{2} - a \iff a=\frac{1}{4}\,\,\text{or}\,\,(1-4a)^{\frac{p-1}{2}}=-1.$$

This means, 

$$a\,\,\text{is a fixed point of}\,\,D_{p+2}(1,x)\iff a=\frac{1}{4}\,\,\text{or}\,\,(1-4a)^{\frac{p-1}{2}}=-1.$$

In other words, 

$$a\,\,\text{is a fixed point of}\,\,D_{p+2}(1,x)\iff a=\frac{1}{4}=\Big(\frac{1}{2}\Big)^2\,\,\text{or}\,\,(1-4a)\,\,\text{is a quadratic nonresidue}.$$

Define 

$$A:=\Big\{a\in \mathbb{F}_p\,\,\Big |\,\,a=\frac{(1-\alpha)(1+\alpha)}{4},\,\,\alpha\in \mathbb{F}_p^{\times}\Big\}.$$

Clearly, the sets $A$, $\{\frac{1}{4}\}$, and the set of all fixed points form a partition of $\mathbb{F}_p$, and the cardinality of $A$ is $\frac{p-1}{2}$. 

Let $a\in A$. Then we have 
$$D_{p+2}(1,a) = \frac{1}{2} (1-4a)^\frac{p+1}{2} - a + \frac{1}{2}=1-3a.$$

Let $f(a) = 1-3a$. Then
$$f^{(n)}(a) = \frac{1-(-3)^n}{4} + (-3)^n a.$$

The rest of the proof is similar to that of Theorem~\ref{TT1}. 

\end{proof} 


\subsection{The case $a=0$}\label{FKA} 

In this subsection, we consider the case $a=0$. The permutation behaviour of $D_n(0,x)$ was explained in \cite[Theorem~2.1]{NF}. Let $a=0$ in the explicit expressions for $D_n(1,x)$ and $D_n^{\prime}(1,x)$ to obtain

\[
D_n(0,x) = \left\{
\begin{alignedat}{2}
  & 0 & \quad & \text{if}\ n\,\,\text{is odd},\\
  & 2(-x)^l  & \quad & \text{if}\ n=2l.\\
\end{alignedat}
\right.
\]

and

\[
D_n^{\prime}(0,x) = \left\{
\begin{alignedat}{2}
  & 0 & \quad & \text{if}\ n\,\,\text{is odd},\\
  & -2l(-x)^{l-1} & \quad & \text{if}\ n=2l.\\
\end{alignedat}
\right.
\]

The following proposition is an immediate result of the fact that $D_n(0,x)=0$ in even characteristic. 

\begin{prop}
$D_n(0,x)$ is never a permutation polynomial over $\mathbb{Z}_{2^t}$. 
\end{prop}

The following proposition is an immediate result of the fact that $D_n^{\prime}(0,0)=0$ in odd characteristic whenever $n\neq 2$, and $D_2(0,x)=-2x$ and $D_2^{\prime}(0,x)=-2$. 

\begin{prop}
Let $p$ be odd. $D_n(0,x)$ is never a permutation polynomial over $\mathbb{Z}_{p^t}$ whenever $n\neq 2$. $D_2(0,x)$ is a permutation polynomial over $\mathbb{Z}_{p^t}$ for any $p$.
\end{prop}


\subsection{Permutation behaviour over $\mathbb{Z}_{2^t}$}

In this subsection, we explain the permutation behaviour of RDPs over $\mathbb{Z}_{2^t}$. To do so, we first need to consider the permutation behaviour of RDPs over $\mathbb{Z}_{2}$ which has been explained in \cite[Table 1]{Dobbertin} and \cite[Theorem 2.5]{HL}. We present it in the following lemma. 

\begin{lem}
$D_n(1,x)$ is a PP over $\mathbb{Z}_2$ if and only if $n>0$ and $n\equiv 0\pmod{3}$. 
\end{lem}

\begin{rmk}
There are no CPPs over $\mathbb{Z}_2$ as $f(x)$ and $f(x)+x$ cannot be one-to-one on $\mathbb{Z}_2$ at the same time. 
\end{rmk}

From the recurrence relation 

$$D_0(1,x)=2, \,\,D_1(1,x)=1,\,\,D_n(1,x)=D_{n-1}(1,x)-xD_{n-2}(1,x)\,\,\text{for}\,\,n\geq 2,$$

we obtain the sequence $\{D_n(1,1)\}_{n=0}^{\infty}$ that also appeared in \cite{HMSY}. 

$$2, 1, -1, -2, -1, 1, 2, 1, -1, -2, -1, 1, 2, 1, -1, -2, \ldots$$

In modulo $2$, we get 

$$0, 1, 1, 0, 1, 1, 0, 1, 1, 0, 1, 1, 0, 1, 1, 0, 1, 1, 0, 1, 1, 0, \ldots ,$$

which is a sequence with period $3$. We have

\[
D_n(1,1) = \left\{
\begin{alignedat}{2}
  & 0 & \quad & \text{if}\ n\equiv 0 \pmod{3},\\
  & 1    & \quad & \text{if}\ n\equiv 1\,\,\textnormal{or}\,\,2\pmod{3}.\\
\end{alignedat}
\right.
\]

Since $D_0(1,0)=2$ and $D_1(1,0)=1$, we get the following sequence when $x=0$:

\[
D_n(1,0) = \left\{
\begin{alignedat}{2}
  & 0 & \quad & \text{if}\ n=0,\\
  & 1    & \quad & \text{if}\ n\geq 1.\\
\end{alignedat}
\right.
\]

Even though we state the following lemma in this subsection, it is true for any prime $p$. 

\begin{lem}\label{lll1}
Let $p$ be a prime and $n_1, n_2>1$. If $n_1\equiv n_2\pmod{p(p^2-1)}$, then $\displaystyle{D_{n_1}^{\prime}(1,x)=D_{n_2}^{\prime}(1,x)}.$
\end{lem}

\begin{proof}
For each $x\in \mathbb{F}_p$, there exists $y\in \mathbb{F}_{p^2}$ such that $x=y(1-y)$. Then we have 
$$D_{n}(1,x)=D_{n}(1,y(1-y))=y^n+(1-y)^n.$$
Since $x=y(1-y)$, we have $\displaystyle{\frac{dy}{dx}=\frac{-1}{2y-1}}$. 
The proof follows from the fact that
\[
\begin{split} 
D_{n_1}^{\prime}(1,x)&=n_1\,(y^{n_1-1}-\,(1-y)^{n_1-1})\cdot \displaystyle{\Big(\frac{-1}{2y-1}\Big)}\cr
&=n_2\,(y^{n_2-1}-\,(1-y)^{n_2-1})\cdot \displaystyle{\Big(\frac{-1}{2y-1}\Big)}\cr
&=D_{n_2}^{\prime}(1,x).
\end{split}
\]
whenever $n_1\equiv n_2\pmod{p(p^2-1)}.$
\end{proof} 

We show that in even characteristic, the regular derivative of the $n$th reversed Dickson polynomial is the $(n-1)$st reversed Dickson polynomial. Note that in even characteristic, $\displaystyle{\frac{dy}{dx}=\frac{-1}{2y-1}=1}$.

\begin{lem}\label{LLJ2} 
Let $n>1$. In even characteristic, $D_n^{\prime}(1,x)=D_{n-1}(1,x)$ whenever $n$ is odd. In particular, $D_n^{\prime}(1,1)$ is odd whenever $n\equiv 3\pmod{6}$.
\end{lem}

\begin{proof}

Since $D_n(1,y(1-y))=y^n+(1-y)^n$, we have 
$$D_n^{\prime}(1,y(1-y))=n\,y^{n-1}+n\,(1-y)^{n-1}=y^{n-1}+(1-y)^{n-1}=D_{n-1}(1,y(1-y)).$$ 

Thus, we have $D_n^{\prime}(1,x)=D_{n-1}(1,x)$. When $n\equiv 3\pmod{6}$, we have $$D_n^{\prime}(1,1)\equiv D_{n-1}(1,1)\equiv 1\pmod{2}.$$

\end{proof}

We state the main theorem in this subsection. 

\begin{thm}
$D_n(1,x)$ is a PP over $\mathbb{Z}_{2^t}$ if and only if $n\equiv 3\pmod{6}$.
\end{thm}

\begin{proof}
Assume that $n\equiv 3\pmod{6}$. Then, we have $D_n(1,0)=1$ and $D_n(1,1)=0$, which implies that $D_n(1,x)$ is a PP over $\mathbb{Z}_2$. Let $s$ be an integer. Then, in modulo 2, we have 

\[
\begin{split}
D_n^{\prime}(1,s)&=\sum_{i=1}^{\lfloor n/2\rfloor}\,\frac{ni}{n-i}\binom{n-i}{i}\,s^{i-1}\cr
\end{split}
\]

If $s$ is even, then $D_n^{\prime}(1,s)\equiv 1\pmod{2}.$ If $s$ is odd, then we have 

\[
\begin{split}
D_n^{\prime}(1,s)&\equiv D_n^{\prime}(1,1)\pmod{2}\cr
&\equiv D_{n-1}(1,1)\pmod{2}\cr
&\equiv 1\pmod{2}
\end{split}
\]

Therefore, $D_n^{\prime}(1,s)\not\equiv 0$ for every integer $s$. By Lemma~\ref{L2}, $D_n(1,x)$ is a PP over $\mathbb{Z}_{2^t}$. We have now shown that if $n\equiv 3\pmod{6}$, then $D_n(1,x)$ is a PP over $\mathbb{Z}_{2^t}$.

Now Assume that $D_n(1,x)$ is a PP over $\mathbb{Z}_{2^t}$. Then $D_n(1,x)$ is a PP over $\mathbb{Z}_{2}$ which implies $n\equiv 0\pmod{3}$. We show that $n$ cannot be even. If $n$ is even, then $n\equiv 0\pmod{6}$, and $D_n^{\prime}(1,s)=0$ for any even integer $s$, which is a contradiction. Thus $n\equiv 3\pmod{6}$. 
\end{proof}

\subsection{Permutation behaviour over $\mathbb{Z}_{3^t}$}

In this subsection, we consider the permutation behaviour of RDPs of the first kind over $\mathbb{Z}_{3^t}$. From the recurrence relation 

$$D_0(1,x)=2, \,\,D_1(1,x)=1,\,\,D_n(1,x)=D_{n-1}(1,x)-xD_{n-2}(1,x)\,\,\text{for}\,\,n\geq 2,$$

we obtain the sequence $\{D_n(1,1)\}_{n=0}^{\infty}\pmod{3}$:

$$2, 1, 2, 1, 2, 1, 2, 1, 2, 1, 2, 1, 2, 1, 2, 1, 2, 1 \ldots $$

We have

\begin{equation}\label{SJ1} 
D_n(1,1) = \left\{
\begin{alignedat}{2}
  & 1 & \quad & \text{if}\ n\equiv 1 \pmod{2},\\
  & 2    & \quad & \text{if}\ n\equiv 0 \pmod{2}.\\
\end{alignedat}
\right.
\end{equation}

When $x=0$, we get the following sequence:

\begin{equation}\label{SJ2}
D_n(1,0) = \left\{
\begin{alignedat}{2}
  & 2 & \quad & \text{if}\,\,n = 0,\\
  & 1    & \quad & \text{if}\,\,n\geq 1.\\
\end{alignedat}
\right.
\end{equation} 

When $x=-1$, we obtain

$$2, 1, 0, 1, 1, 2, 0, 2, 2, 1, 0, 1, 1, 2, 0, 2, 2, 1, \ldots ,$$

which gives us

\begin{equation}\label{SJ3} 
D_n(1,-1) = \left\{
\begin{alignedat}{2}
  & 0 & \quad & \text{if}\ n\equiv 2\,\,\text{or}\,\,6 \pmod{8},\\
  & 1    & \quad & \text{if}\ n\equiv 1, 3,\,\,\text{or}\,\,4 \pmod{8},\\
  & 2  & \quad & \text{if}\ n\equiv 0, 5, \,\,\text{or}\,\,7 \pmod{8}.
\end{alignedat}
\right.
\end{equation} 

To study the permutation behaviour of RDPs over $\mathbb{Z}_{3^t}$, we first need to know their permutation behaviour over $\mathbb{Z}_3$, which is given by the following Lemma that appeared in \cite[Corollary 5.2]{HMSY}. The proof of the lemma follows from the three sequences $D_n(1,0), D_n(1,1)$ and $D_n(1,-1)$.

\begin{lem}\label{ll3}
$D_n(1,x)$ is a PP over $\mathbb{Z}_3$ if and only if $n\equiv 2\pmod{4}$. 
\end{lem}

\begin{rmk}
Clearly,  $D_0(1,x)$ is a CPP over $\mathbb{Z}_3$.
\end{rmk}

In the next Theorem, we give necessary and sufficient conditions for $D_n(1,x)$ to a CPP over $\mathbb{Z}_3$. 

\begin{thm}\label{TJ191}
Let $n\geq 1$. Then, $D_n(1,x)$ is a CPP over $\mathbb{Z}_3$ if and only if $n\equiv 2, 6\pmod{8}$. 
\end{thm}

\begin{proof}

Consider the following three sequences $D_n(1,0)$, $D_n(1,1) + 1$, and $D_n(1,-1) +2$. 

\[
D_n(1,0) = \left\{
\begin{alignedat}{2}
  & 2 & \quad & \text{if}\,\,n = 0,\\
  & 1    & \quad & \text{if}\,\,n\geq 1.\\
\end{alignedat}
\right.
\]

\[
D_n(1,1) + 1= \left\{
\begin{alignedat}{2}
  & 2 & \quad & \text{if}\ n\equiv 1 \pmod{2},\\
  & 0   & \quad & \text{if}\ n\equiv 0 \pmod{2}.\\
\end{alignedat}
\right.
\]

\[
D_n(1,-1) +2 = \left\{
\begin{alignedat}{2}
  & 2 & \quad & \text{if}\ n\equiv 2\,\,\text{or}\,\,6 \pmod{8},\\
  & 0    & \quad & \text{if}\ n\equiv 1, 3,\,\,\text{or}\,\,4 \pmod{8},\\
  & 1  & \quad & \text{if}\ n\equiv 0, 5, \,\,\text{or}\,\,7 \pmod{8}.
\end{alignedat}
\right.
\]

Clearly, $D_n(1,x)+x$ is a PP over $\mathbb{Z}_3$ if and only if $n\equiv 1,2,3,6\pmod{8}$. The proof follows from the fact that $D_n(1,x)$ is a PP over $\mathbb{Z}_3$ if and only if $n\equiv 2\pmod{4}$. 

\end{proof}

\begin{rmk}\label{r1}
Since $$\displaystyle{D_n(1,x)=\sum_{i=0}^{\lfloor n/2\rfloor}\,\frac{n}{n-i}\binom{n-i}{i}\,(-x)^i},$$
we have 
$$\displaystyle{D_n^{\prime}(1,x)=\sum_{i=1}^{\lfloor n/2\rfloor}\,\frac{ni}{n-i}\binom{n-i}{i}\,(-1)^i\,x^{i-1}},$$
and $\displaystyle{D_n^{\prime}(1,0)=-n}$.
\end{rmk} 

\begin{lem}\label{LLLJ4}
Let $p$ be an odd prime, $s\in \mathbb{F}_p\setminus \{\frac{1}{4}\}$, and $n\geq 2$. Then
\begin{center}
$D_n^{\prime}(1,s)=0$ if and only if $n\equiv 0\pmod{p}$ or $E_{n-2}(1,s)=0$,
\end{center} 
where $D_n^{\prime}(1,x)$ is the derivative of the $n$th reversed Dickson polynomial of the first kind and $E_{n-2}(1,x)$ is the $(n-2)$nd reversed Dickson polynomial of the second kind. 
\end{lem} 

\begin{proof}
For each $x\in \mathbb{F}_p$, there exists $y\in \mathbb{F}_{p^2}$ such that $x=y(1-y)$. Then we have 
$$D_{n}(1,x)=D_{n}(1,y(1-y))=y^n+(1-y)^n.$$
Since $x=y(1-y)$, we have $\displaystyle{\frac{dy}{dx}=\frac{-1}{2y-1}}$. When $y\neq \frac{1}{2}$, i.e.  $x\neq \frac{1}{4}$, we have
\[
\begin{split} 
D_{n}^{\prime}(1,x)&=n\,(y^{n-1}-\,(1-y)^{n-1})\cdot \displaystyle{\Big(\frac{-1}{2y-1}\Big)}\cr
&= - n\,\,\,\displaystyle{\frac{y^{n-1}-\,(1-y)^{n-1}}{2y-1}}\cr
&= - n \,\,E_{n-2}(1,x).
\end{split}
\]
Thus, we have the desired result. 
\end{proof} 

\begin{lem}\label{LLJ4} Let $p=3$. Then the sequence $\{D_n(1,1)\}_{n=0}^{\infty}$ is given by  
\[
D_{n}^{\prime}(1,\frac{1}{4}) = D_{n}^{\prime}(1,1) = \left\{
\begin{alignedat}{2}
  & 0 & \quad & \text{if}\ n \equiv 3,4,6,7,9,10,12,13,15,16,18,19,21,22 \pmod{24}\\
  & 1 & \quad & \text{if}\ n \equiv 2,8,14,20 \pmod{24}\\
  & 2 & \quad & \text{if}\ n \equiv 5,11,17,23 \pmod{24};
\end{alignedat}
\right.
\]
\end{lem} 

\begin{proof}
The proof follows from the sequence~\ref{B0} in Appendix~\ref{AppB}. 
\end{proof} 

Now we state the main theorem in this subsection. 

\begin{thm}\label{T3}
$D_n(1,x)$ is a PP over $\mathbb{Z}_{3^t}$ if and only if $n\equiv 2, 14\pmod{24}$.
\end{thm}

\begin{proof} 

Assume that $n\equiv 2\,\text{or}\,\,14\pmod{24}$. Then, $n\equiv2\pmod{4}$. By Lemma~\ref{ll3}, $D_n(1,x)$ is a PP over $\mathbb{Z}_3.$ Now we show that $D_n'(1,s)\neq0$ for all $s\in\mathbb{Z}_3$. Remark~\ref{r1} says that $D_n'(1,0)\neq0$ whenever $n\equiv 2\,\text{or}\,\,14\pmod{24}$. Lemma~\ref{LLJ4} says that $D_n'(1,1)\neq0$ whenever $n\equiv 2\,\text{or}\,\,14\pmod{24}$. From Lemma~\ref{LLLJ4} and \eqref{EJ4}, we have $D_n'(1,-1)\neq0$ whenever $n\equiv 2\,\text{or}\,\,14\pmod{24}$. We have now shown that $D_n'(1,s)\neq0$ for all $s\in\mathbb{Z}_3$.

Now assume that $D_n(1,x)$ is a PP over $\mathbb{Z}_{3^t}.$ Then, $D_n(1,x)$ is a PP over $\mathbb{Z}_3$ and $D_n'(1,s)\neq0$ for all $s\in\mathbb{Z}_3$. Since $D_n(1,x)$ is a PP over $\mathbb{Z}_3$, $n\equiv2\pmod{4}$. Because of Lemma~\ref{LJ1} (ii) and Lemma~\ref{lll1}, we only need to consider the polynomials whose indices are less than $24$. Remark~\ref{r1} says that $D_6^{\prime}(1,0)=0=D_{18}^{\prime}(1,0)$, and Lemma~\ref{LLJ4} implies that $D_{10}^{\prime}(1,1)=0=D_{22}^{\prime}(1,1)$. Thus, we have the desired result. 
\end{proof} 


\subsection{Complete permutation polynomials over $\mathbb{Z}_p$}

\begin{prop}
Let $p$ be an odd prime. If $n = 2, 2p$, then $D_{n}(1,x)$ is a CPP of $\mathbb{F}_p$
\end{prop} 

\begin{proof}
Since $2$ and $2p$ belong to the same $p$-cyclotomic coset modulo $p$, for all $x\in \mathbb{F}_p$, we have $$D_{2}(1,x) = D_{2p}(1,x) = 1+(p-2)x,$$
and 
$$D_{2}(1,x) +x = D_{2p}(1,x) + x= 1+(p-1)x.$$
The rest of the proof follows from the fact that every linear polynomial in $\mathbb{F}_p[x]$ is a PP of $\mathbb{F}_p$. 
\end{proof} 

\begin{prop}
Let $p>3$ be an odd prime. If $n = 3, 3p$, then $D_{n}(1,x)$ is a CPP of $\mathbb{F}_p$
\end{prop}

\begin{proof}
Since $3$ and $3p$ belong to the same $p$-cyclotomic coset modulo $p$, for all $x\in \mathbb{F}_p$, we have 
$$D_{3}(1,x) = D_{3p}(1,x) = 1-(p-3)x,$$ 
and
$$D_{3}(1,x) +x = D_{3p}(1,x) + x= 1+(p-2)x.$$
The rest of the proof follows from the fact that every linear polynomial in $\mathbb{F}_p[x]$ is a PP of $\mathbb{F}_p$. 
\end{proof}

In \cite{HMSY}, the authors found sufficient conditions on $n$ for the polynomial $D_n(1,x)$ to be a PP of $\mathbb{F}_p$ when  $p \equiv 1\pmod{12}$, $p \equiv 5\pmod{12}$ or  $p \equiv 7\pmod{12}$. We present their results in the following three lemmas. 

\begin{lem}\label{L1}
Let $p \equiv 1\pmod{12}$ or $p \equiv 5\pmod{12}$. If $n = p+1$, then $D_{n}(1,x)$ is a PP of $\mathbb{F}_p$.
\end{lem}

\begin{lem}\label{L3}
Let $p \equiv 1\pmod{12}$ or $p \equiv 7\pmod{12}$. If $n = p+2$, then $D_{n}(1,x)$ is a PP of $\mathbb{F}_p$.
\end{lem}

\begin{lem}
Let $p \equiv 1\pmod{12}$ or $p \equiv 7\pmod{12}$. If $n = 2p+1$, then $D_{n}(1,x)$ is a PP of $\mathbb{F}_p$.
\end{lem}

\begin{rmk}
Since $np=(2p)=2p^2+p \equiv 2+p\pmod{p^2-1},$ $p+2$ and $2p+1$ belong to the same $p$-cyclotomic coset modulo $p^2-1$. Thus, $D_{2p+1}(1,x)$ is a PP of $\mathbb{F}_p$ if and only if $D_{2p+1}(1,x)$ is a PP of $\mathbb{F}_p$. Moreover,
 $D_{p+2}(1,x)=D_{2p+1}(1,x)$ for all $x\in \mathbb{F}_p$.
\end{rmk}

The natural question to ask is `` Are all the permutation polynomials in the aforementioned lemmas complete permutation polynomials of $\mathbb{F}_p$ as well? " It turns out that not all of the polynomials are complete permutation polynomials of $\mathbb{F}_p$. 

\begin{prop}\label{L2}
Let $p \equiv 1\pmod{12}$. If $n = p+1$, then $D_{n}(1,x)$ is a CPP of $\mathbb{F}_p$.
\end{prop}

\begin{proof}

From Lemma~\ref{L1}, we know that $D_{p+1}(1,x)$ is a PP of $\mathbb{F}_p$. We only need to show that $D_{p+1}(1,x) + x$ is a PP of $\mathbb{F}_p$. From \cite[Proposition 5.1]{HMSY}, we have $$D_{p+1}(1,x)=\frac{1}{2}+\frac{1}{2}(1-4x)^{\frac{p+1}{2}},$$

which implies
\begin{align*}
D_{p+1}(1,x) + x &= \frac{1}{2}+\frac{1}{2}(1-4x)^{\frac{p+1}{2}} +x \\
&= \left(\frac{1}{2}+\frac{1}{2}x^{\frac{p+1}{2}}+\frac{1}{4} - \frac{x}{4}\right)\circ (1-4x).
\end{align*}

Thus, $D_{p+1}(1,x) + x$ is a PP of $\mathbb{F}_p$ if and only if $\displaystyle{f(x)=2x^\frac{p+1}{2} - x}$ is a PP of $\mathbb{F}_p$. We consider the following partition of $\mathbb{F}_p$:

$$\mathbb{F}_p = \{0\} \cup (\mathbb{F}_p^{\times})^2 \cup (\mathbb{F}_p^{\times} \setminus (\mathbb{F}_p^{\times})^2),$$

where $(\mathbb{F}_p^{\times})^2$ is the set of quadratic residues in $\mathbb{F}_p$ and $\mathbb{F}_p^{\times} \setminus (\mathbb{F}_p^{\times})^2$ is the set of quadratic nonresidues in $\mathbb{F}_p$. 

Note that $f(0) = 0$. Let $\alpha \in (\mathbb{F}_p^{\times})^2$,
Then, $\alpha = \beta^2$, where $\beta \neq 0$.
\begin{align*}
    f(\alpha) &= \alpha (2\alpha^\frac{p-1}{2} - 1)\\
    &= \alpha(2(\beta^2)^\frac{p-1}{2} - 1)\\
    &= \alpha(2\beta^{(p-1)} - 1)\\
    &= \alpha(2\cdot 1 - 1)\\
    &= \alpha
\end{align*}

This shows that $f$ maps $(\mathbb{F}_p^{\times})^2$ to itself. Since $f$ is the identity map on $(\mathbb{F}_p^{\times})^2$, $f$ is clearly one-to-one on $(\mathbb{F}_p^{\times})^2$.

Let $\gamma \in \mathbb{F}_p^{\times}\setminus (\mathbb{F}_p^{\times})^2$.
\begin{align*}
    f(\gamma) &= \gamma (2\alpha^\frac{p-1}{2} - 1)\\
    &=\gamma(-2-1)\\
    &= -3\gamma
\end{align*}
Since $p\equiv 1\pmod{12}$, $3$ and $-1$ are quadratic residues in $\mathbb{F}_p$, which implies $-3\gamma$ is a quadratic nonresidue in $\mathbb{F}_p$. This shows that $f$ maps $\mathbb{F}_p^{\times} \setminus (\mathbb{F}_p^{\times})^2$ to itself, and it is clearly one-to-one on $\mathbb{F}_p^{\times} \setminus (\mathbb{F}_p^{\times})^2$. Thus, $f$ is a one to one on $\mathbb{F}_p$.
Therefore, we have shown that $D_{p+1}(1,x)+x$ is a PP on $\mathbb{F}_p$, and thus $D_{p+1}(1,x)$ is a CPP on $\mathbb{F}_p$.
\end{proof}

\begin{prop}
Let $p \equiv 5\pmod{12}$. If $n = p+1$, then $D_{n}(1,x)$ is not a CPP of $\mathbb{F}_p$. 
\end{prop}

\begin{proof} 
Consider the partition

$$\mathbb{F}_p = \{0\} \cup (\mathbb{F}_p^{\times})^2 \cup \mathbb{F}_p^{\times} \setminus (\mathbb{F}_p^{\times})^2.$$

Since $p\equiv 5\pmod{12}$, $3$ is a quadratic nonresidue and $-1$ is a quadratic residue. This implies that $f$ maps both $(\mathbb{F}_p^{\times})^2$ and $\mathbb{F}_p^{\times} \setminus (\mathbb{F}_p^{\times})^2$ to $(\mathbb{F}_p^{\times})^2$. Thus, $f$ is not one-to-one on $\mathbb{F}_p$.
Therefore, $D_{p+1}(1,x)$ is not a CPP on $\mathbb{F}_p$. 
\end{proof}

\begin{prop}\label{L4}
Let $p \equiv 1\pmod{12}$. If $n = p+2$, then $D_{n}(1,x)$ is CPP of $\mathbb{F}_p$.
\end{prop}

\begin{proof}
From \cite[Proposition 5.1]{HMSY}, we have 

$$D_{p+2}(1,x)= \frac{1}{2}+\frac{1}{2}(1-4x)^\frac{p+1}{2}-x.$$

Then,  
\begin{align*}
D_{p+2}(1,x) + x
&= \frac{1}{2}+\frac{1}{2}(1-4x)^{\frac{p+1}{2}}\\
&=D_{p+1}(1,x)
\end{align*}
The proof follows from the fact that $D_{p+1}(1,x)$ is a PP of $\mathbb{F}_p$ whenever $p \equiv 1\pmod{12}$.
\end{proof}

\begin{prop}\label{L5}
Let $p \equiv 7\pmod{12}$. If $n = p+2$, then $D_{n}(1,x)$ is not CPP of $\mathbb{F}_p$.
\end{prop}

\begin{proof}
The proof follows from the fact that $D_{p+2}(1,x) + x=D_{p+1}(1,x)$ and $D_{p+1}(1,x)$ is not a PP of $\mathbb{F}_p$ whenever $p\equiv 7\pmod{12}$. 
\end{proof}

\begin{rmk}
Since $D_{p+2}(1,x)=D_{2p+1}(1,x)$ for all $x\in \mathbb{F}_p$. We have the following:
\begin{enumerate}
\item Let $p \equiv 1\pmod{12}$. If $n = 2p+1$, then $D_{n}(1,x)$ is CPP of $\mathbb{F}_p$.
\item Let $p \equiv 7\pmod{12}$. If $n = 2p+1$, then $D_{n}(1,x)$ is not CPP of $\mathbb{F}_p$.
\end{enumerate} 
\end{rmk}

We now state a conjecture on complete permutation polynomials arising from reversed Dickson polynomials of the first kind that is strongly supported by our computer search results. Conjecture~\ref{CJ10C} is a consequence of Conjecture~\ref{CCC1} and the results derived in this subsection. 

\begin{conj}\label{CJ10C}
Let $p$ be an odd prime and let $1\leq n \leq p^2-1.$ Then $D_n(1,x)$ is a CPP on $\mathbb{F}_p$ if and only if 

\[
n = \left\{
\begin{alignedat}{2}
  & 2,2p,3,3p,p+1,p+2,2p+1 & \quad & \text{if}\ p\equiv 1 \pmod{12},\\
  & 2,2p,3,3p              & \quad & \text{if}\ p\equiv 3 \pmod{4} \,\,\textnormal{or}\ p\equiv 5 \pmod{12}.
\end{alignedat}
\right.
\]
\end{conj}


\section{Reversed Dickson polynomials of the second kind}\label{S4}

\subsection{Background}

The $n$th reversed Dickson polynomial of the second kind is given by $$\displaystyle{E_n(1,x)=\sum_{i=0}^{\lfloor n/2\rfloor}\,\binom{n-i}{i}\,(-x)^i},$$

which is the solution of the recurrence relation $$E_0(1,x)=1,\,E_1(1,x)=1,\,E_n(1,x)=E_{n-1}(1,x)-xE_{n-2}(1,x)\,\,\text{for}\,n\geq 2.$$

Let $x\in \mathbb{F}_p$. There exists $y\in \mathbb{F}_{p^2}$ such that $x=y(1-y)$. The functional expression of $E_n(1,x)$ is given by 

$$\displaystyle{E_n(1,x)=E_n(1,y(1-y))=\frac{y^{n+1}-(1-y)^{n+1}}{2y-1}},$$

where $y\neq \frac{1}{2}$. When $y=\frac{1}{2}$, that is $x=\frac{1}{4}$, we have $E_n(1,\frac{1}{4})=\frac{n+1}{2^n}$; see \cite[Eq. 2.6]{NF}. 

\begin{rmk}\label{RJ2} 
The explicit expression of the regular derivative of the reversed Dickson polynomials of the second kind is given by 
$$\displaystyle{E_n^{\prime}(1,x)=\sum_{i=1}^{\lfloor n/2\rfloor}\,i\,\binom{n-i}{i}\,(-1)^{i}\,x^{i-1}}.$$
\end{rmk} 


\subsection{When is $E_{n_1}^{\prime}(1,x)=E_{n_2}^{\prime}(1,x)$ for all $x\in \mathbb{F}_p$?}

In order to study the permutation behaviour of $E_n(1,x)$ over $\mathbb{Z}_{p^t}$, we need to consider the derivative of the polynomial $E_{n}^{\prime}(1,x)$. In this subsection, we study the sequence $\{E_n^{\prime}(1,x)\}_{n=0}^{\infty}$. Our results indicate that it is worth noting that the periodicity of the sequence $\displaystyle{E_n^{\prime}\Big(1,\frac{1}{4}\Big)}$ depends on whether $p$ is a Mersenne prime or non-Mersenne prime. 

\begin{lem}(\cite[Section 2]{NF})\label{LJ2} 
Let $p$ be an odd prime. If $n_1\equiv n_2\pmod{p^2-1}$, then $E_{n_1}(1,x)=E_{n_2}(1,x)$ for all $x\in \mathbb{F}_p\setminus \{\frac{1}{4}\}$.  
\end{lem} 

\begin{cor}\label{ll1}
Let $p$ be an odd prime. If $n_1\equiv n_2\pmod{p(p^2-1)}$, then $E_{n_1}(1,x)=E_{n_2}(1,x)$ for all $x\in \mathbb{F}_p$.  
\end{cor}

\begin{proof}
The proof follows from Lemma~\ref{LJ2} and the fact $$\displaystyle{E_{n+p(p^2-1)}\Big(1,\frac{1}{4}\Big)=\frac{n+p(p^2-1)+1}{2^{n+p(p^2-1)}}=\frac{n+1}{2^n}=E_n\Big(1,\frac{1}{4}\Big)}.$$
\end{proof}

\begin{rmk}
The result in Corollary~\ref{ll1} can be generalized to reversed Dickson polynomials of the $(k+1)$-th kind for any $\mathbb{F}_q$, where $q$ is a power of an odd prime $p$. Let $\mathbb{F}_q$ be the finite field with $q$ elements. If $n_1\equiv n_2\pmod{p(q^2-1)}$, then $D_{n_1,k}(1,x)=D_{n_2,k}(1,x)$ for all $x\in \mathbb{F}_q$. When $k=0$, the modulus can be reduced to $q^2-1$ as 

$$\displaystyle{D_{n,0}\Big(1,\frac{1}{4}\Big)=D_n\Big(1,\frac{1}{4}\Big)=\frac{1}{2^{n-1}}}.$$

\end{rmk}


\subsubsection{Periodicity of $E_n^{\prime}(1,x)$ over $\mathbb{F}_p$} 

\begin{lem}\label{j181}
Let $p$ be an odd prime and $n_1,n_2>1$. If $n_1\equiv n_2\pmod{p(p^2-1)}$, then $E_{n_1}^{\prime}(1,x)=E_{n_2}^{\prime}(1,x)$ for all $x\in \mathbb{F}_p\setminus \{\frac{1}{4}\}$.
\end{lem}

\begin{proof}
For each $x\in \mathbb{F}_p$, there exists $y\in \mathbb{F}_{p^2}$ such that $x=y(1-y)$. The functional expression of $E_n(1,x)$ is given by 

$$\displaystyle{E_n(1,x)=E_n(1,y(1-y))=\frac{y^{n+1}-(1-y)^{n+1}}{2y-1}},$$

where $y\neq \frac{1}{2}$, i.e. $x\neq \frac{1}{4}$; see \cite[Section 2]{NF}. Since $(2y-1)^2=1-4x$, its derivative $E_n^{\prime}(1,x)$ is given by 

\[
\begin{split}
E_n^{\prime}(1,x)&=\frac{(2y-1)\,(n+1)\,(y^n+(1-y)^n)-\,2\,(y^{n+1}-(1-y)^{n+1}}{(2y-1)^2}\cdot \displaystyle{\frac{-1}{2y-1}}\cr
&=\frac{-1}{(2y-1)^3}\,\Big((2y-1)\,(n+1)\,(y^n+(1-y)^n)\Big)\cr
&+ \frac{2}{(2y-1)^3}\,(y^{n+1}-(1-y)^{n+1})\cr
&=\frac{-1}{(1-4x)}\,\cdot (n+1)\cdot D_n(1,x)+ \frac{2}{(1-4x)}\,\cdot \,E_n(1,x).
\end{split} 
\]

Thus the proof follows from the fact that 

\[
\begin{split}
E_{n+p(p^2-1)}^{\prime}(1,x)&=\frac{-1}{(1-4x)}\,\cdot (n+p(p^2-1)+1)\cdot D_{n+p(p^2-1)+1}(1,x)\cr
&+ \frac{2}{(1-4x)}\,\cdot \,E_{n+p(p^2-1)+1}(1,x)\cr
&= \frac{-1}{(1-4x)}\,\cdot (n+1)\cdot D_n(1,x)+ \frac{2}{(1-4x)}\,\cdot \,E_n(1,x).
\end{split}
\]

\end{proof} 

\begin{cor}\label{CCJJ}
In $\mathbb{F}_p$, if $n+1\equiv 0\pmod{p}$, then $E_{n}^{\prime}(1,x)=\frac{2}{(1-4x)}\,\,E_{n}(1,x)$, where $x\in \mathbb{F}_p\setminus \{\frac{1}{4}\}$. 
\end{cor} 

\begin{proof}
The proof follows from Lemma~\ref{j181}. 
\end{proof} 

\begin{cor}\label{JJJC2}
In $\mathbb{F}_p$, we have 

\[
\begin{split}
E_{np}^{\prime}(1,x)&=\frac{-1}{(1-4x)}\,\cdot D_{np}(1,x)+ \frac{2}{(1-4x)}\,\cdot \,E_{np}(1,x),
\end{split} 
\]\end{cor} 

\begin{proof}
The proof follows from Lemma~\ref{j181}. 
\end{proof} 

\begin{cor}\label{JJJC3}
In $\mathbb{F}_p$, if $n+1\equiv 0\pmod{p^2-1}$, then $$E_{n}^{\prime}(1,x)=\frac{-1}{(1-4x)}\,\cdot (n+1)\cdot \,D_{n}(1,x),$$ where $x\in \mathbb{F}_p\setminus \{0, \frac{1}{4}\}$. 
\end{cor} 

\begin{proof}
The proof follows from Lemma~\ref{j181}.  
\end{proof} 

We now explore the periodicity of $\displaystyle{E_n^{\prime}\Big(1,\frac{1}{4}\Big)}$. 


\subsubsection{$p=3$.}

\begin{lem}\label{j182}
Let $p=3$. Then 

\[
\displaystyle{E_n^{\prime}\Big(1,\frac{1}{4}\Big)} =E_n^{\prime}(1,1)= \left\{
\begin{alignedat}{2}
  & 0 & \quad & \text{if}\ n\equiv 0,1,8,9,10,17 \pmod{18},\\
  & 1     & \quad & \text{if}\ n\equiv 3, 6, 11, 13, 14, 16\pmod{18},\\
  & 2     & \quad & \text{if}\ n\equiv 2,4,5,7, 12, 15 \pmod{18}.
\end{alignedat}
\right.
\]

\end{lem}

\begin{proof}
The proof follows from the sequence \ref{B5} in Appendix~\ref{AppB}. 
\end{proof} 

\begin{thm}\label{TJ2}
Let $p=3$ and $n_1,n_2 >1$. If $n_1\equiv n_2\pmod{72}$, then for all $x\in \mathbb{Z}_3$, we have $E_{n_1}^{\prime}(1,x)=E_{n_2}^{\prime}(1,x).$
\end{thm}

\begin{proof}
The proof follows from Lemma~\ref{j181}, Lemma~\ref{j182}, and the fact that $\textnormal{lcm}(24,18)=72$. 
\end{proof} 


\subsubsection{$p=5$.}

\begin{lem}\label{j183}
Let $p=5$. Then 

\[
\displaystyle{E_n^{\prime}\Big(1,\frac{1}{4}\Big)} =E_n^{\prime}(1,-1)= \left\{
\begin{alignedat}{2}
  & 0 & \quad & \text{if}\ n\equiv 0,1,4,5,6,9,10,11,14,15,16,19 \pmod{20},\\
  & 1     & \quad & \text{if}\ n\equiv 12,18\pmod{20},\\
  & 2     & \quad & \text{if}\ n\equiv 7,13 \pmod{20},\\
  & 3     & \quad & \text{if}\ n\equiv 3,17\pmod{20},\\
  & 4     & \quad & \text{if}\ n\equiv 2,8 \pmod{20}.\\
\end{alignedat}
\right.
\]
\end{lem}

\begin{proof}
The proof follows from the sequence \ref{B6} in Appendix~\ref{AppB}.  
\end{proof} 

\begin {thm}
Let $p=5$ and $n_1,n_2 >1$. If $n_1\equiv n_2\pmod{120}$, then for all $x\in \mathbb{Z}_5$, we have $E_{n_1}^{\prime}(1,x)=E_{n_2}^{\prime}(1,x).$
\end{thm}

\begin{proof}
The proof follows from Lemma~\ref{j181}, Lemma~\ref{j183}, and the fact that $\textnormal{lcm}(120,20)=120$. 
\end{proof} 


\subsubsection{$p=7$.}

\begin{lem}\label{j184}
Let $p=7$. Then $\displaystyle{E_n^{\prime}\Big(1,\frac{1}{4}\Big)}$ is a sequence with period 21, and 

\[
\displaystyle{E_n^{\prime}\Big(1,\frac{1}{4}\Big)} =E_n^{\prime}(1,2)\left\{
\begin{alignedat}{2}
  & 0 & \quad & \text{if}\ n\equiv 0,1,6,7,8,13,14,15,20\pmod{21},\\
  & 1     & \quad & \text{if}\ n\equiv 4,5,\pmod{21},\\
  & 2     & \quad & \text{if}\ n\equiv 18,19\pmod{21},\\
  & 3     & \quad & \text{if}\ n\equiv 9,17\pmod{21},\\
  & 4     & \quad & \text{if}\ n\equiv 11,12\pmod{21},\\
  & 5    & \quad & \text{if}\ n\equiv 3,16\pmod{21},\\
  & 6     & \quad & \text{if}\ n\equiv 2,10,\pmod{21}.
\end{alignedat}
\right.
\]
\end{lem}

\begin{proof}
The proof follows from the sequence \ref{B7} in Appendix~\ref{AppB}. 
\end{proof} 

\begin {thm}
Let $p=7$ and $n_1,n_2 >1$. If $n_1\equiv n_2\pmod{336}$, then for all $x\in \mathbb{Z}_7$, we have $E_{n_1}^{\prime}(1,x)=E_{n_2}^{\prime}(1,x).$
\end{thm}

\begin{proof}
The proof follows from Lemma~\ref{j181}, Lemma~\ref{j184}, and the fact that $\textnormal{lcm}(336,21)=336$. 
\end{proof} 

\subsection{Mersenne primes and the sequence $\displaystyle{\Big\{E_n^{\prime}\Big(1,\frac{1}{4}\Big)\Big\}_{n=1}^{\infty}}$} 

As mentioned earlier, it is interesting to note that the sequence $\displaystyle{E_n^{\prime}\Big(1,\frac{1}{4}\Big)}$ depends on whether $p$ is a Mersenne prime or non-Mersenne prime. We present three conjectures on the sequence $\displaystyle{E_n^{\prime}\Big(1,\frac{1}{4}\Big)}$ that are strongly supported by our computational results. 

\begin{conj}\label{C1}
Let $p>3$ be a non-Mersenne prime. Then, $\displaystyle{E_n^{\prime}\Big(1,\frac{1}{4}\Big)}$ is a sequence with period $p(p-1)$. 
\end{conj}

\begin{conj}\label{C2}
Let $p>3$ be a Mersenne prime. Then, $\displaystyle{E_n^{\prime}\Big(1,\frac{1}{4}\Big)}$ is a sequence with period $p(p-1)/2$. 
\end{conj}

\begin{rmk}
To prove Conjecture~\ref{C1} and Conjecture~\ref{C2}, one has to consider the following sequence:
$$E_n^{\prime}\Big(1,\frac{1}{4}\Big)=\sum_{i=1}^{\lfloor n/2\rfloor}\,i\,\binom{n-i}{i}\,(-1)^i\,\Big(\frac{1}{4}\Big)^i\pmod{p}.$$
The Conjectures~\ref{C1} and \ref{C2} lead to the following immediate conjecture. 
\end{rmk}

\begin{conj}\label{C3}
Let $p>3$ and $n_1,n_2 >1$. If $n_1\equiv n_2\pmod{p(p^2-1)}$, then $E_{n_1}^{\prime}(1,x)=E_{n_2}^{\prime}(1,x)$  for all $x\in \mathbb{Z}_p$.
\end{conj}


\subsection{The case $a=0$}\label{SKA}

In this subsection, we consider the case $a=0$. The permutation behaviour of $E_n(0,x)$ was explained in \cite[Theorem~2.1]{NF}. Let $a=0$ in the explicit expressions for $E_n(1,x)$ and $E_n^{\prime}(1,x)$ to obtain

\[
E_n(0,x) = \left\{
\begin{alignedat}{2}
  & 0 & \quad &{\rm if}\,n\,{\rm is\,\,odd},\\
  & (-x)^l & \quad & {\rm if}\,n=2l.\\
\end{alignedat}
\right.
\]

and 

\[
E_n^{\prime}(0,x) = \left\{
\begin{alignedat}{2}
  & 0 & \quad &{\rm if}\,n\,{\rm is\,\,odd},\\
  & - l(-x)^{l-1} & \quad & {\rm if}\,n=2l.\\
\end{alignedat}
\right.
\]

The following proposition is an immediate result of the fact that $E_n^{\prime}(0,0)=0$ whenever $n\neq 2$, and $E_2(0,x)=-x$ and $E_2^{\prime}(0,x)=-1$. 

\begin{prop}
Let $p$ be a prime. $E_n(0,x)$ is never a permutation polynomial over $\mathbb{Z}_{p^t}$ whenever $n\neq 2$. $E_2(0,x)$ is a permutation polynomial over $\mathbb{Z}_{p^t}$ for any $p$.
\end{prop}


\subsection{Permutation behaviour over $\mathbb{Z}_{2^t}$}

In this subsection, we explain the permutation behaviour of RDPs of the second kind over $\mathbb{Z}_{2^t}$. From the recurrence relation 

$$E_0(1,x)=1,\,E_1(1,x)=1,\,E_n(1,x)=E_{n-1}(1,x)-xE_{n-2}(1,x)\,\,\text{for}\,n\geq 2,$$

we obtain the sequence $\{E_n(1,1)\}_{n=0}^{\infty}$: 

$$1, 1, 0, -1, -1, 0, 1, 1, 0, -1, -1, 0, 1, 1, 0, -1, -1, 0, \ldots$$

Taking modulo $2$, we obtain 

$$1, 1, 0, 1, 1, 0, 1, 1, 0, 1, 1, 0, 1, 1, 0, 1, 1, 0, 1, 1, 0, \ldots ,$$

which is a sequence with period $3$. Thus, we have 

\[
E_n(1,1) = \left\{
\begin{alignedat}{2}
  & 0 & \quad & {\rm if}\,n\equiv 2\,\pmod{3},\\
  & 1 & \quad & {\rm if}\,n\equiv 0\,\,\textnormal{or}\,\,1\pmod{3}.\\
\end{alignedat}
\right.
\]

Since $E_0(1,x)=1$ and $E_1(1,x)=1$, we have $E_n(1,0)=1$ for all $n$.  

\begin{lem}
$E_n(1,x)$ is a PP over $\mathbb{Z}_2$ if and only if $n\equiv 2\pmod{3}$. 
\end{lem}

\begin{proof} 
The proof follows from the sequences for $E_n(1,0)$ and $E_n(1,1)$. 
\end{proof} 

\begin{lem}\label{LLJ1} 
In even characteristic, we have 
$$E_n^{\prime}(1,x)=D_{n+1}^{\prime}(1,x).$$
In particular, if $n$ is even, 
$$E_n^{\prime}(1,x)=D_{n}(1,x).$$
Moreover,  
$$\displaystyle{E_{n_1}^{\prime}(1,x)=E_{n_2}^{\prime}(1,x)}$$
whenever $n_1\equiv n_2\pmod{p(p^2-1)}$. 
\end{lem} 

\begin{proof}
Let $x\in \mathbb{F}_p$, $y\in \mathbb{F}_{p^2}$, and $x=y(1-y)$. Then, in even characteristic, the functional expression of $E_n(1,x)$ is given by 
$$E_n(1,x)=E_n(1,y(1-y))=y^{n+1}+(1-y)^{n+1}=D_{n+1}(1,x).$$
Thus, $E_n^{\prime}(1,x)=D_{n+1}^{\prime}(1,x)$. Since
$$E_n^{\prime}(1,x)=E_n^{\prime}(1,y(1-y))=(n+1)\,(y^{n}+(1-y)^{n})=(n+1)\,D_{n}(1,x),$$
we have $E_n^{\prime}(1,x)=D_{n}(1,x)$ whenever $n$ is even. The rest of the proof follows from Lemma~\ref{lll1}. 
\end{proof} 

We now present the main result in this subsection. 

\begin{thm}
$E_n(1,x)$ is a PP over $\mathbb{Z}_{2^t}$ if and only if $n\equiv 2\pmod{6}$.
\end{thm}

\begin{proof}
Assume that $n \equiv 2 \pmod{6}$. Since $n\equiv 2 \pmod{6}$, we have $n\equiv 2 \pmod{3}$, which means $E_n(1,x)$ is a PP over $\mathbb{Z}_2$. We will show that $E_n^{\prime}(1,s) \neq 0$ for $s\in \mathbb{Z}_2$. Since $E_2^{\prime}(1,x) = -1\neq 0$ for all $x\in \mathbb{Z}_2$, we have $E_n^{\prime}(1,x)\neq 0$ for all $x\in \mathbb{Z}_2$ whenever $n\equiv 2 \pmod{6}$ because of Lemma~\ref{LLJ1}. We have shown that if $n\equiv 2\pmod{6}$, then $E_n(1,x)$ is a PP over $\mathbb{Z}_{2^t}$. 

Now assume that $E_n(1,x)$ is a PP over $\mathbb{Z}_{2^t}$. Then $E_n(1,x)$ is a PP over $\mathbb{Z}_{2}$ which implies $n \equiv 2 \pmod{3}$.
Also, we know that $E_n^{\prime}(1,s) \neq 0$ for all $s \in \mathbb{Z}_{2}$. Since $E_n^{\prime}(1,0) = n - 1$, $n\not\equiv 5 \pmod{6}$. Hence, we have the desired result. 
\end{proof}


\subsection{Permutation behaviour over $\mathbb{Z}_{3^t}$}

In this subsection, we explain the permutation behaviour of RDPs of the second kind over $\mathbb{Z}_{3^t}$. From the recurrence relation 

$$E_0(1,x)=1,\,E_1(1,x)=1,\,E_n(1,x)=E_{n-1}(1,x)-xE_{n-2}(1,x)\,\,\text{for}\,n\geq 2,$$

we obtain the sequence $\{E_n(1,1)\}_{n=0}^{\infty}$ which also appeared in \cite{HQZ}: 

$$1, 1, 0, -1, -1, 0, 1, 1, 0, -1, -1, 0, 1, 1, 0, -1, -1, 0, \ldots$$

In modulo $3$, we obtain 

$$1, 1, 0, 2, 2, 0, 1, 1, 0, 2, 2, 0, 1, 1, 0, 2, 2, 0, 1, 1, 0, \ldots ,$$

which is a sequence with period $6$. We have

\[
E_n(1,1) = \left\{
\begin{alignedat}{2}
  & 0 & \quad & \text{if}\ n\equiv 2, 5\pmod{6}, \\
  & 1 & \quad & \text{if}\ n\equiv 0, 1\pmod{6}, \\
  & 2 & \quad & \text{if}\ n\equiv 3, 4\pmod{6}.
\end{alignedat}
\right.
\]

Since $E_0(1,x)=1$ and $E_1(1,x)=1$, we have $E_n(1,0)=1$ for all $n$.  

In a similar way, we can obtain the sequence $\{E_n(1,-1)\}_{n=0}^{\infty}$ in modulo $3$: 

$$1, 1, 2, 0, 2, 2, 1, 0, 1, 1, 2, 0, 2, 2, 1, 0, 1, 1, 2, 0, 2, 2, 1, 1, \ldots ,$$

which is a sequence with period $8$. We have

\begin{equation}\label{EJ4} 
E_n(1,-1) = \left\{
\begin{alignedat}{2}
  & 0 & \quad & \text{if}\ n\equiv 3, 7\pmod{8}, \\
  & 1 & \quad & \text{if}\ n\equiv 0, 1, 6\pmod{8}, \\
  & 2 & \quad & \text{if}\ n\equiv 2, 4, 5\pmod{8}.
\end{alignedat}
\right.
\end{equation}

\begin{lem}\label{l3}
$E_n(1,x)$ is a PP over $\mathbb{Z}_3$ if and only if $n\equiv 2, 3, 5, 15, 20\pmod{24}$. 
\end{lem}

\begin{proof}
Assume that $n\equiv 2, 3, 5, 15,\,\,\textnormal{or}\,\,20\pmod{24}$. Then, it follows from the sequences of $E_n(1,0)$, $E_n(1,1)$ and $E_n(1,-1)$ that $E_n(1,x)$ is a PP over $\mathbb{Z}_3$. 

Now assume that $E_n(1,x)$ is a PP over $\mathbb{Z}_3$. Then, clearly $$n\not\equiv 0, 1, 6, 7, 8, 9, 12, 13, 14, 16, 17, 18, 19, 22 \pmod{24}. $$

By comparing the values of $E_n(1,1)$ and $E_n(1,-1)$ that 
$$n\not\equiv 4, 10, 11, 21, 23\pmod{24}.$$

This completes the proof. 
\end{proof}

\begin{thm}\label{Jj192}
$E_n(1,x)$ is a CPP over $\mathbb{Z}_3$ if and only if $n\equiv 3,15\pmod{24}$. 
\end{thm}

\begin{proof}
Let's consider the following three sequences $E_n(1,0), E_n(1,1)+1$ and $E_n(1,-1)+2$: 

$$E_n(1,0)=1\,\,\,\,\text{for all}\,\,\,\,n.$$ 

\[
E_n(1,1) + 1= \left\{
\begin{alignedat}{2}
  & 1 & \quad & \text{if}\ n\equiv 2, 5\pmod{6}, \\
  & 2 & \quad & \text{if}\ n\equiv 0, 1\pmod{6}, \\
  & 0 & \quad & \text{if}\ n\equiv 3, 4\pmod{6}.
\end{alignedat}
\right.
\]

\[
E_n(1,-1) + 2 = \left\{
\begin{alignedat}{2}
  & 2 & \quad & \text{if}\ n\equiv 3, 7\pmod{8}, \\
  & 0 & \quad & \text{if}\ n\equiv 0, 1, 6\pmod{8}, \\
  & 1 & \quad & \text{if}\ n\equiv 2, 4, 5\pmod{8}.
\end{alignedat}
\right.
\]

Clearly, $E_n(1,x)+x$ is a PP over $\mathbb{Z}_3$ if and only if $n\equiv 0,1,3,6,15\pmod{24}$. The proof follows from the fact that $E_n(1,x)$ is a PP over $\mathbb{Z}_3$ if and only if $n\equiv 2, 3, 5, 15,20\pmod{24}$. 
\end{proof}

\begin{rmk}
The third reversed Dickson polynomial of the second kind $E_3(1,x)$ is a CPP for any $p\geq 3$ since $E_3(1,x)=1-2x$ and $E_3(1,x)+x=1-x$ are both PPs of $\mathbb{F}_p$. 
\end{rmk} 

We explain the permutation behaviour of the reversed Dickson polynomials of the second kind over $\mathbb{Z}_3$ in the following theorem. We note to the reader that because Corollary~\ref{ll1} and Theorem~\ref{TJ2}, we only need to consider RDPs whose index $n$ is less than 72. Let $S=\{2,3,5,15,20,29,39,50,51,68\}$. 

\begin{thm}\label{l4}
$E_n(1,x)$ is a PP over $\mathbb{Z}_{3^t}$ if and only if $n\in S$. 
\end{thm}

\begin{proof}

Assume that $n\in S$. Then, $n\equiv 2, 3, 5, 15, 20\pmod{24},$ which says that $E_n(1,x)$ is a PP over $\mathbb{Z}_3$. We show that $E_n^{\prime}(1,s)\neq0$ for all $s\in\mathbb{Z}_3$ whenever $n\in S$. 

From the explicit expression for derivative of the reversed Dickson polynomials of the second kind in Remark~\ref{RJ2}, we have $E_n^{\prime}(1,0)=-(n-1)$. In particular, $E_{n}^{\prime}(1,0)\neq 0$ whenever $n\in S$. It follows from Lemma~\ref{j182} that $E_n^{\prime}(1,1)\neq 0$ whenever $n\in S$. We only need to show that $E_n^{\prime}(1,-1)\neq 0$ whenever $n\in S$. 
Because $E_{n_1}^{\prime}(1,x)=E_{n_2}^{\prime}(1,x)$ for all $x\in \mathbb{F}_3\setminus\{1\}$ whenever $n_1\equiv n_2\pmod{24}$, we have $E_2^{\prime}(1,-1)=E_{50}^{\prime}(1,-1), E_3^{\prime}(1,-1)=E_{51}^{\prime}(1,-1), E_5^{\prime}(1,-1)=E_{29}^{\prime}(1,-1), E_{15}^{\prime}(1,-1)=E_{39}^{\prime}(1,-1)$, and $E_{20}^{\prime}(1,-1)=E_{68}^{\prime}(1,-1)$.
Since 
$$\displaystyle{E_2(1,x)=1-x,\,\,E_3(1,x)=1-2x,\,\,E_5(1,x)=1-x},$$
all three derivatives $E_2^{\prime}(1,-1), E_3^{\prime}(1,-1)$ and $E_5^{\prime}(1,-1)$ are nonzero. Thus, $E_{50}^{\prime}(1,-1)$, $E_{51}^{\prime}(1,-1)$, $E_{29}^{\prime}(1,-1)$ are nonzero as well. Now we show that $E_{15}^{\prime}(1,-1)$ and $E_{20}^{\prime}(1,-1)$ are nonzero which implies $E_{39}^{\prime}(1,-1)$ and $E_{68}^{\prime}(1,-1)$ are also nonzero. 

From Corollary~\ref{JJJC2}, Corollary~\ref{JJJC3} and \eqref{SJ3}, we have $E_{15}^{\prime}(1,-1)=2$. From Corollary~\ref{CCJJ} and \eqref{EJ4}, we have $E_{20}^{\prime}(1,-1)=2$. We have now shown that $E_n(1,x)$ is a PP over $\mathbb{Z}_3$ and $E_n'(1,x)\neq0$ for $x\in\mathbb{Z}_3$ whenever $n\in s.$ Thus, $E_n(1,x)$ is a PP over $\mathbb{Z}_{3^t}.$

Assume that $E_n(1,x)$ is PP over $\mathbb{Z}_{3^t}.$ Because it is a PP over $\mathbb{Z}_3,$ we have $n\equiv2,3,5,15,20\pmod{24}.$ By Theorem~\ref{TJ2}, we have $E_{n_1}'(1,x)=E_{n_2}'(1,x)$ for all $x\in\mathbb{Z}_3$ whenever $n_1\equiv n_2\pmod{72}.$ Therefore, to complete the proof, we only need to show that $E_{26}(1,x)$, $E_{27}(1,x)$, $E_{44}(1,x)$, $E_{53}(1,x)$, and $E_{63}(1,x)$, are not PPs over $\mathbb{Z}_{3^t}$. 

Because $E_n(1,x)$ is PP over $\mathbb{Z}_{3^t}$, $E_n'(1,x)\neq0$ for all $x\in\mathbb{Z}_3.$ However, Lemma~\ref{j182} says that 

$$E_{26}^{\prime}(1,1)=E_{27}^{\prime}(1,1)=E_{44}^{\prime}(1,1)=E_{53}^{\prime}(1,1)=E_{63}^{\prime}(1,x)=0.$$

This completes the proof. 
\end{proof}


\subsection{Permutation behaviour over $\mathbb{Z}_{p}$, where $p>3$}

In this subsection, we present sufficient conditions on the index $n$ for the polynomial $E_n(1,x)$ to be a PP of $\mathbb{F}_p$. Our numerical results strongly suggest that these conditions on $n$ are also necessary. Therefore, we end this subsection with three conjectures on the PP and CPP behaviour of the polynomial $E_n(1,x)$ over $\mathbb{F}_p$.

\begin{thm}
If $n \equiv 2,3,15,94 \pmod{120}$, then $E_{n}(1,x)$ is a PP over $\mathbb{Z}_5$. 
\end{thm}

\begin{proof}
We first compute the polynomials $E_2(1,x)$, $E_3(1,x)$, $E_{15}(1,x)$ and $E_{94}(1,x)$. It is straightforward from the recurrence relation that 
$$E_{2}(1,x)= 1 - x\,\,\,\,\text{and}\,\,\,\,E_{3}(1,x)= 1 - 2x,$$
and they are clearly PPs of $\mathbb{Z}_{5}$.

We compute the polynomial $E_{15}(1,x)$. When $x=\frac{1}{4}$, i.e. $y=\frac{1}{2}$, we have 
$$E_{15}(1,\frac{1}{4})=E_{15}(1,-1)=\frac{15+1}{2^{15}}=\frac{1}{2^3}=2.$$

Let $x=y(1-y)$. Then the the functional expression of $E_n(1,x)$ is given by 

$$\displaystyle{E_n(1,x)=E_n(1,y(1-y))=\frac{y^{n+1}-(1-y)^{n+1}}{2y-1}},$$

where $y\neq \frac{1}{2}$, i.e. $x\neq \frac{1}{4}$. Here, $y$ is in the extension field $\mathbb{F}_{5^2}$. From the functional expression, we have $E_{15}(1,0)=1$. Let $x\in \mathbb{F}_5\setminus \{0,-1\}$. Then 

\[
\begin{split}
E_{15}(1,x)&=E_{15}(1,y(1-y))\cr
&=\frac{y^{16}-(1-y)^{16}}{2y-1}\cr
&=\frac{y^{-8}-(1-y)^{-8}}{2y-1}\cr
&=\frac{(1-y)^{8}-y^8}{(y(1-y))^8\,2y-1}\cr
&=-\,\frac{((1-y)^2+y^2)\cdot ((1-y)^4+y^4)}{(y(1-y))^8}\cr
&=-\,\frac{D_2(1,x)\cdot D_4(1,x)}{x^8}\cr
&=-\,D_2(1,x)\cdot D_4(1,x)\cr
&=4x^3+x-1.
\end{split}
\]

Thus, the polynomial $E_{15}(1,x)\in \mathbb{F}_5[x]$ is given by

\[
E_{15}(1,x) = \left\{
\begin{alignedat}{2}
  & 1& \quad & \text{if}\ x=0, \\
  & 2 & \quad & \text{if}\ x=-1, \\
  & 4x^3+x-1 & \quad & \text{if}\ x\in \mathbb{F}_5\setminus \{0,-1\}.
\end{alignedat}
\right.
\]

Clearly, the cubic polynomial $4x^3+x-1$ permutes the elements of the set $x\in \mathbb{F}_5\setminus \{0,-1\}$. Thus, $E_{15}(1,x)$ is a PP of $\mathbb{F}_5$. 

We have thus far shown that the polynomials $E_{94}(1,x)$. $E_{2}(1,x)$, $E_{3}(1,x)$ and $E_{15}(1,x)$ are permutation polynomials of $\mathbb{F}_5$. Now we consider the last case where $n=94$. 

When $x=\frac{1}{4}$, i.e. $y=\frac{1}{2}$, we have 
$$E_{94}(1,\frac{1}{4})=E_{15}(1,-1)=\frac{94+1}{2^{15}}=0.$$ 
From the functional expression, we have $E_{15}(1,0)=1$. Let $x\in \mathbb{F}_5\setminus \{0,-1\}$. Then 

\[
\begin{split}
E_{94}(1,x) &= \frac{y^{95} - (1-y)^{95}}{2y-1}\cr
&= \frac{y^{-1} - (1-y)^{-1}}{2y-1}\cr
&= \frac{-1}{y(1-y)}\cr
&= -x^{-1}\cr
&= -x^3.
\end{split}
\]

Therefore, the polynomial $E_{94}(1,x)\in \mathbb{F}_5[x]$ is given by

\[
E_{94}(1,x) = \left\{
\begin{alignedat}{2}
  & 1& \quad & \text{if}\ x=0, \\
  & 0 & \quad & \text{if}\ x=-1, \\
  & -x^3 & \quad & \text{if}\ x\in \mathbb{F}_5\setminus \{0,-1\}.
\end{alignedat}
\right.
\]

Clearly, the monomial $-x^3$ permutes the elements of the set $x\in \mathbb{F}_5\setminus \{0,-1\}$. Thus, $E_{94}(1,x)$ is a PP of $\mathbb{F}_5$. This completes the proof. 
\end{proof}

\begin{thm}
If $n\equiv 2, 3, 170 \pmod{336}$, then $E_n(1,x)$ is a PP over $\mathbb{Z}_7$. 
\end{thm}

\begin{proof}
It is straightforward from the recurrence relation that $$E_{2}(1,x)= 1 - x\,\,\,\,\text{and}\,\,\,\,E_{3}(1,x)= 1 - 2x,$$ and they are clearly PPs of $\mathbb{Z}_{7}$. We compute the polynomial $E_{170}(1,x)$. When $x=\frac{1}{4}$, i.e. $y=\frac{1}{2}$, we have 
$$E_{170}(1,\frac{1}{4})=E_{170}(1,2)=\frac{170+1}{2^{170}}=\frac{3}{4}=6.$$

Let $x=y(1-y)$. Then the the functional expression of $E_n(1,x)$ is given by 

$$\displaystyle{E_n(1,x)=E_n(1,y(1-y))=\frac{y^{n+1}-(1-y)^{n+1}}{2y-1}},$$

where $y\neq \frac{1}{2}$, i.e. $x\neq \frac{1}{4}$. Here, $y$ is in the extension field $\mathbb{F}_{5^2}$. From the functional expression, we have $E_{170}(1,0)=1$. Let $x\in \mathbb{F}_7\setminus \{0,-1\}$. Since $(2y-1)^2=1-4x$, we have

\[
\begin{split} 
E_{170}(1,x)&=\frac{y^{171}-(1-y)^{171}}{2y-1} \cr
&= \frac{y^{-21}-(1-y)^{-21}}{2y-1} \cr
&=\frac{(1-y)^{21}-y^{21}}{(y(1-y))^{21}\,(2y-1)}\cr
&=\frac{(1-y^7)^{3}-(y^7)^{3}}{(y(1-y))^{21}\,(2y-1)}\cr
&=-\,\frac{(2y-1)^6\,(1-y^7(1-y)^7)}{(y(1-y))^{21}}\cr
&= -\,\frac{(1-4x)^3\,(1-x^7)}{x^{21}}\cr
&=-\,\frac{(1-4x)^3\,(1-x^7)}{x^{3}}\cr
&=3x^5+6x^4+6x^3+6x.
\end{split}
\]

Thus, the polynomial $E_{170}(1,x)\in \mathbb{F}_7[x]$ is given by

\[
E_{170}(1,x) = \left\{
\begin{alignedat}{2}
  & 1 & \quad & \text{if}\ x=0, \\
  & 6 & \quad & \text{if}\ x=2,\\
  & 3x^5+6x^4+6x^3+6x & \quad & \text{if}\ x\in\mathbb{F}_7\backslash\{0,2\}.
\end{alignedat}
\right.
\]

Clearly, the polynomial $3x^5+6x^4+6x^3+6x$ permutes the elements of the set $x\in\mathbb{F}_7\backslash\{0,2\}$. Thus, $E_{170}(1,x)$ is a PP of $\mathbb{F}_7$. 

\end{proof}

\begin{conj}\label{conjJ51}
Let $p=5$. Then, $E_{n}(1,x)$ is a PP over $\mathbb{F}_p$ if and only if $n \equiv 2,3,15,94 \pmod{120}$.
\end{conj}

\begin{conj}
Let $p=7$. Then, $E_n(1,x)$ is a PP over $\mathbb{F}_p$ if and only If $n\equiv 2, 3, 170 \pmod{336}$. 
\end{conj}

\begin{conj}\label{conjJ52}
Let $p>7$. Then, $E_n(1,x)$ is a PP of $\mathbb{F}_p$ if and only if $n\equiv 2, 3 \pmod{p(p^2-1)}$. 
\end{conj}

\begin{conj}\label{conjJ53}
Let $p\geq 5$. Then, $E_n(1,x)$ is a CPP of $\mathbb{F}_p$ if and only if $n\equiv 3 \pmod{p(p^2-1)}$. 
\end{conj}

\begin{rmk}
We have shown that $E_2(1,x)$ and $E_3(1,x)$ are PPs of $\mathbb{F}_p$ for any $p\geq 3$. Clearly, $E_2(1,x)+x$ is not a PP of $\mathbb{F}_p$ for any $p$. Note that $E_{15}(1,x)+x$ is not a PP of $\mathbb{F}_5$ as 
$$E_2(1,0)=E_2(1,-1)-1,$$

and $E_{95}(1,x)+x$ is not a PP of $\mathbb{F}_5$ as 
$$E_{95}(1,-1)-1=E_2(1,2)+2.$$

Also, $E_{170}(1,x)+x$ is not a PP of $\mathbb{F}_7$ as 
$$E_2(1,0)=E_2(1,1)+1.$$

Since the polynomial  $E_3(1,x)+x$ is a PP of $\mathbb{F}_p$ for any $p$, proving Conjectures~\ref{conjJ51},\ref{conjJ52} and \ref{conjJ53} proves Conjecture~\ref{conjJ53}.
\end{rmk} 

\begin{rmk}
Cycle type of the permutation polynomial $E_3(1,x)$ for $p>3$ is described in Theorem~\ref{TT1}.
\end{rmk}

\begin{rmk}
Cycle type of the permutation polynomial $E_3(1,x)$ for $p=3$ is described in Theorem~\ref{TJ7T1}.
\end{rmk}

\begin{rmk}
Cycle type of the permutation polynomial $E_2(1,x)$ for $p\geq 3$ is described in Theorem~\ref{TJ7T2}.
\end{rmk}


\section{Proof of a special case of a conjecture on reversed Dickson PPs ($p=5$)}\label{S5}

In this subsection, we confirm the Conjecture~\ref{CCC1} for $p=5$. In \cite{HMSY}, the authors proved that the values in the statement of the following theorem are sufficient for the $D_n(1,x)$ to be a permutation polynomial of $\mathbb{F}_p$. To confirm the conjecture, we only need to prove that those values are also necessary.

\begin{thm}\label{T5} 
$D_{n}(1,x)$ is a PP over $\mathbb{Z}_5$ if and only if $n = 2,3,6,10,15 \pmod{24}$.
\end{thm}

\begin{proof} 

Assume that $n\not\equiv 2,3,6,10,15\pmod{24}.$ From the sequence in Table 1 in Appendix~\ref{AppB}, we have 

\begin{center}
\[
D_{n}(1,0) = \left\{
\begin{alignedat}{2}
  & 2 & \quad & \text{if}\ n = 0,\\
  & 1 & \quad & \text{if}\ n \geq 1.
\end{alignedat}
\right.
\] 
\newline
\[
D_{n}(1,1) = \left\{
\begin{alignedat}{2}
  & 1 & \quad & \text{if}\ n \equiv 1, 5 \pmod{6},\\
  & 2 & \quad & \text{if}\ n \equiv 0  \pmod{6},\\
  & 3 & \quad & \text{if}\ n \equiv 3  \pmod{6},\\
  & 4 & \quad & \text{if}\ n \equiv 2, 4 \pmod{6}.
\end{alignedat}
\right.
\]
\newline
\[
D_{n}(1,2) = \left\{
\begin{alignedat}{2}
  & 0 & \quad & \text{if}\ n \equiv 3,9,15,21 \pmod{24},\\
  & 1 & \quad & \text{if}\ n \equiv 1,4,5,18,20 \pmod{24},\\
  & 2 & \quad & \text{if}\ n \equiv 0,2,7,10,11 \pmod{24},\\
  & 3 & \quad & \text{if}\ n \equiv 12,14,19,22,23 \pmod{24},\\
  & 4 & \quad & \text{if}\ n \equiv 6,8,13,16,17 \pmod{24}.
\end{alignedat}
\right.
\]
\newline
\[
D_{n}(1,-2) = \left\{
\begin{alignedat}{2}
  & 0 & \quad & \text{if}\ n \equiv 2 \pmod{4},\\
  & 1 & \quad & \text{if}\ n \equiv 1  \pmod{4},\\
  & 2 & \quad & \text{if}\ n \equiv 0,3  \pmod{4}.
\end{alignedat}
\right.
\]
\newline
\[
D_{n}(1,-1) = \left\{
\begin{alignedat}{2}
  & 1 & \quad & \text{if}\ n \equiv 1 \pmod{4},\\
  & 2 & \quad & \text{if}\ n \equiv 0  \pmod{4},\\
  & 3 & \quad & \text{if}\ n \equiv 2  \pmod{4},\\
  & 4 & \quad & \text{if}\ n \equiv 3  \pmod{4}.
\end{alignedat}
\right.
\]
\end{center}

Clearly, $D_0(1,x)$ is not a PP of $\mathbb{Z}_5$. Therefore, for the rest of the proof, assume that $n\geq 1$. 

Since $D_{n}(1,0) = 1$ for all $n \geq 1$, $D_n(1,x)$ is clearly not a PP over $\mathbb{Z}_5$ whenever $n\equiv 1,5 \pmod{6}$, $n\equiv 1 \pmod{4}$, and $n\equiv 4, 18, 20 \pmod{24}$. If we exclude the $n$ values less than or equal to $23$ that satisfy these congruences, we have $n = 8,12,14,16,22$.

Since $D_8(1,x), D_{12}(1,x)$, and $D_{16}(1,x)$ map $3$ and $4$ to $2$, and $D_{14}(1,x)$ and $D_{22}(1,x)$ map $2$ and $4$ to $3$, $D_n(1,x)$ is not a PP of $\mathbb{Z}_5$ whenever $n = 8,12,14,16,22$. 

Also, note that $D_{24}(1,x)$ is not a PP since $D_{24}(1,1)=2=D_{24}(1,1)$.

Since $D_{n_1}(1,x)=D_{n_2}(1,x)$ for all $x\in \mathbb{F}_p$ whenever $n_1,n_2>0$ are integers such that $n_1\equiv n_2\pmod{p^2-1}$, we have shown that if $n\not\equiv 2,3,6,10,15\pmod{24}$, then $D_n(1,x)$ is not a PP of $\mathbb{Z}_5$. This completes the proof. 
 
 \end{proof}

\section{Proof of a special case of a conjecture on reversed Dickson PPs ($p=7$)}\label{S6}

In this subsection, we confirm the Conjecture~\ref{CCC1} for $p=7$. In \cite{HMSY}, the authors proved that the values in the statement of the following theorem are sufficient for the $D_n(1,x)$ to be a permutation polynomial of $\mathbb{F}_p$. To confirm the conjecture, we only need to prove that those values are also necessary. 

\begin{thm}\label{T7} 
$D_{n}(1,x)$ is a PP over $\mathbb{Z}_7$ if and only if $n = 2,3,9,14,15 ,21\pmod{48}$.
\end{thm}

\begin{proof}

Assume that $n \not\equiv 2,3,9,14,15 ,21\pmod{48}$. From the sequence  in Table 2 in Appendix~\ref{AppB}, we have 

\begin{center}
\[
D_{n}(1,0) = \left\{
\begin{alignedat}{2}
  & 2 & \quad & \text{if}\ n = 0,\\
  & 1 & \quad & \text{if}\ n \geq 1.
\end{alignedat}
\right.
\] 
\newline
\[
D_{n}(1,1) = \left\{
\begin{alignedat}{2}
  & 1 & \quad & \text{if}\ n \equiv 1, 5 \pmod{6},\\
  & 2 & \quad & \text{if}\ n \equiv 0  \pmod{6},\\
  & 5 & \quad & \text{if}\ n \equiv 3  \pmod{6},\\
  & 6 & \quad & \text{if}\ n \equiv 2, 4 \pmod{6}.
\end{alignedat}
\right.
\]
\newline
\[
D_{n}(1,2) = \left\{
\begin{alignedat}{2}
  & 1 & \quad & \text{if}\ n \equiv 1 \pmod{3},\\
  & 2 & \quad & \text{if}\ n \equiv 0 \pmod{3},\\
  & 4 & \quad & \text{if}\ n \equiv 2 \pmod{3}.
\end{alignedat}
\right.
\]
\newline
\[
D_{n}(1,3) = \left\{
\begin{alignedat}{2}
  & 0 & \quad & \text{if}\ n \equiv 4,12,20,28,36,44 \pmod{48},\\
  & 1 & \quad & \text{if}\ n \equiv 1,7,27,32,34,45,46 \pmod{48},\\
  & 2 & \quad & \text{if}\ n \equiv 0,2,13,14,17,23,43 \pmod{48},\\
  & 3 & \quad & \text{if}\ n \equiv 5,6,9,15,35,40,42 \pmod{48},\\
  & 4 & \quad & \text{if}\ n \equiv 11,16,18,29,30,33,39 \pmod{48},\\
  & 5 & \quad & \text{if}\ n \equiv 19,24,26,37,38,41,47 \pmod{48},\\
  & 6 & \quad & \text{if}\ n \equiv 3,8,10,21,22,25,31 \pmod{48}.
\end{alignedat}
\right.
\]
\newline
\[
D_{n}(1,4) = \left\{
\begin{alignedat}{2}
  & 0 & \quad & \text{if}\ n \equiv 2,6,10,14,18,22 \pmod{24},\\
  & 1 & \quad & \text{if}\ n \equiv 1,7,8  \pmod{24},\\
  & 2 & \quad & \text{if}\ n \equiv 0,17,23  \pmod{24},\\
  & 3 & \quad & \text{if}\ n \equiv 3,4,21 \pmod{24},\\
  & 4 & \quad & \text{if}\ n \equiv 9,15,16 \pmod{24},\\
  & 5 & \quad & \text{if}\ n \equiv 5,11,12 \pmod{24},\\
  & 6 & \quad & \text{if}\ n \equiv 13,19,20 \pmod{24}.
\end{alignedat}
\right.
\]
\newline
\[
D_{n}(1,5) = \left\{
\begin{alignedat}{2}
  & 0 & \quad & \text{if}\ n \equiv 3 \pmod{6},\\
  & 1 & \quad & \text{if}\ n \equiv 1 \pmod{6},\\
  & 2 & \quad & \text{if}\ n \equiv 0  \pmod{6},\\
  & 3 & \quad & \text{if}\ n \equiv 4, 5  \pmod{6},\\
  & 5 & \quad & \text{if}\ n \equiv 2  \pmod{6}.
\end{alignedat}
\right.
\]
\newline
\[
D_{n}(1,-1) = \left\{
\begin{alignedat}{2}
  & 0 & \quad & \text{if}\ n \equiv 4,12 \pmod{16},\\
  & 1 & \quad & \text{if}\ n \equiv 1, 7 \pmod{16},\\
  & 2 & \quad & \text{if}\ n \equiv 0  \pmod{16},\\
  & 3 & \quad & \text{if}\ n \equiv 2,11,13,14  \pmod{16},\\
  & 4 & \quad & \text{if}\ n \equiv 3,5,6,10  \pmod{16},\\
  & 5 & \quad & \text{if}\ n \equiv 8  \pmod{16},\\
  & 6 & \quad & \text{if}\ n \equiv 9, 15  \pmod{16}.
\end{alignedat}
\right.
\]
\end{center}

Clearly, $D_0(1,x)$ is not a PP of $\mathbb{Z}_7$. Therefore, for the rest of the proof, assume that $n\geq 1$. 
Since $D_{n}(1,0) = 1$ for all $n \geq 1$, $D_n(1,x)$ is clearly not a PP over $\mathbb{Z}_7$ whenever $n\equiv 5 \pmod{6}$, $n\equiv 1 \pmod{3}$, $n\equiv 27, 45 \pmod{48}$, $n\equiv 8 \pmod{24}$, and $n\equiv 1,7 \pmod{16}$. If we exclude the $n$ values less than or equal to $47$ that satisfy these congruences, we have $n = 6,12,18,20,24,26,30,36,38,42,44.$

Note that $D_6(1,x), D_{12}(1,x)$,$D_{18}(1,x), D_{24}(1,x)$, $D_{30}(1,x), D_{36}(1,x)$and $D_{42}(1,x)$ map $1$ and $2$ to $2$. 

Also, $D_{20}(1,x)$ and $D_{44}(1,x)$ map $3$ and $6$ to $0$, and $D_{26}(1,x)$ and $D_{38}(1,x)$ map $3$ and $5$ to $5$. 

Therefore, $D_n(1,x)$ is not a PP of $\mathbb{Z}_5$ whenever $$n = 6,12,18,20,24,26,30,36,38,42,44.$$

Since $D_{48}(1,1)=2=D_{48}(1,2)$, and the fact that  $D_{n_1}(1,x)=D_{n_2}(1,x)$ for all $x\in \mathbb{F}_p$ whenever $n_1,n_2>0$ are integers such that $n_1\equiv n_2\pmod{p^2-1}$, we have shown that if $n \not\equiv 2,3,9,14,15 ,21\pmod{48}$, then $D_n(1,x)$ is not a PP of $\mathbb{Z}_7$. This completes the proof. 
\end{proof}



\appendix

\section{Cycle types of PPs and CPPs}\label{AppA}

In Appendix~\ref{AppA}, we list the cycle types of permutation polynomials and complete permutation polynomials arising from reversed Dickson polynomials. \vskip 0.3in

\begin{tabularx}{\textwidth}{ |X|X|X| }
\hline
& $p=3$ & $p>3$ \\
\hline
$D_2(1,x)\ \ \ \ \ \ \ \ \ $ $D_{2p}(1,x)\ \ \ \ \ \ \ \ \ \ $ $E_3(1,x)\ \ \ \ \ \ \ \ \ $ & $(3)$ & $(\underbrace{\frac{p-1}{j},...,\frac{p-1}{j}}_{j\rm\ times},1),$ where $j\in\mathbb{Z}^+,$ $j\vert p-1,$ and $\text{ord}_p(-2)=\frac{p-1}{j}$ \\
\hline
$D_3(1,x)\ \ \ \ \ \ \ \ \ \ \ \ \ \ \ \ \ $ $D_{3p}(1,x)$ & &$(\underbrace{\frac{p-1}{j},...,\frac{p-1}{j}}_{j\rm\ times},1),$ where $j\in\mathbb{Z}^+,$ $j\vert p-1,$ and $\text{ord}_p(-3)=\frac{p-1}{j}$ \\
\hline
$E_2(1,x)$ & \multicolumn{2}{c|}{$(\underbrace{2,...,2}_{p-1\rm\ times},1)$} \\
\hline
\end{tabularx}

\vskip 0.3in

\begin{tabularx}{\textwidth}{ |X|X| }
\hline
 polynomial & $p\equiv1\pmod{12}$ or $p\equiv7\pmod{12}$ \\
\hline
$D_{p+2}(1,x)\ \ \ \ \ \ \ \ \ \ \ \ \ \ \ \ \ \ \ \ \ \ \ \ \ \ \ \ \ \ \ \ \ \ $ $D_{2p+1}(1,x)$ & $(\underbrace{\frac{p-1}{j},...,\frac{p-1}{j}}_{j/2\rm\ times},\underbrace{1,...,1}_{\frac{p+1}{2}\rm\ times}),$ where $j\in\mathbb{Z}^+,$ $j\vert p-1,$ and $\text{ord}_p(-3)=\frac{p-1}{j}$ \\
\hline
\end{tabularx}

\vskip 0.3in

\begin{tabularx}{\textwidth}{ |X|X|X|X| }
\hline
 polynomial & $p=3$ & $p=5$ & $p=7$ \\
\hline
$E_5(1,x)$ & $(2,1)$ & & \\
\hline
$E_{15}(1,x)$ & $(3)$ & & \\
\hline
$E_{20}(1,x)$ & $(2,1)$ & & \\
\hline
$E_{94}(1,x)$ & & $(3,1,1)$ &  \\
\hline
$E_{170}(1,x)$ & & & $(4,2,1)$ \\
\hline
\end{tabularx}

\vskip 0.3in

\begin{tabularx}{\textwidth}{ |X|X|X| }
\hline
& $p=3$ & $p>3$ \\
\hline
$D_2(1,x)+x\ \ \ \ \ \ \ \ \ $ $D_{2p}(1,x)+x\ \ \ \ \ \ \ \ \ $ $E_3(1,x)+x$ & \multicolumn{2}{c|}{$(\underbrace{2,...,2}_{p-1\rm\ times},1)$} \\
\hline
$D_3(1,x)+x\ \ \ \ \ \ \ \ \ $ $D_{3p}(1,x)+x$ &  & $(\underbrace{\frac{p-1}{j},...,\frac{p-1}{j}}_{j\rm\ times},1),$ where $j\in\mathbb{Z}^+,$ $j\vert p-1,$ and $\text{ord}_p(-2)=\frac{p-1}{j}$ \\
\hline
$E_{15}(1,x)+x$ & $(2,1)$ & \\
\hline
\end{tabularx}


\section{Numerical results}\label{AppB}

In Appendix~\ref{AppB}, we present our numerical results. 

The sequence of $D_{n}^{\prime}(1,1)$ modulo $3$ for $0 \leq n \leq 23$ 

\begin{equation}\label{B0}
D_{n}^{\prime}(1,1): 2,1,1,0,0,2,0,0,1,0,0,2,0,0,1,0,0,2,0,0,0,0,0,2
\end{equation} 

The sequence  of $\displaystyle{E_n^{\prime}\Big(1,\frac{1}{4}\Big)}$ over $\mathbb{Z}_3$ for $0 \leq n \leq 17$ 

\begin{equation}\label{B5}
\displaystyle{E_n^{\prime}\Big(1,\frac{1}{4}\Big)}: 0, 0, 2, 1, 2, 2, 1, 2, 0, 0, 0, 1, 2, 1, 1, 2, 1, 0
\end{equation}

The sequence  of $\displaystyle{E_n^{\prime}\Big(1,\frac{1}{4}\Big)}$ over $\mathbb{Z}_5$ for $0 \leq n \leq 19$

\begin{equation}\label{B6}
\displaystyle{E_n^{\prime}\Big(1,\frac{1}{4}\Big)}: 0, 0, 4, 3, 0, 0, 0, 2, 4, 0, 0, 0, 1, 2, 0, 0, 0, 3, 1, 0
\end{equation}

The sequence of $\displaystyle{E_n^{\prime}\Big(1,\frac{1}{4}\Big)}$ over $\mathbb{Z}_7$ for $0 \leq n \leq 20$

\begin{equation}\label{B7}
\displaystyle{E_n^{\prime}\Big(1,\frac{1}{4}\Big)}: 0, 0, 6, 5, 1, 1, 0, 0, 0, 3, 6, 4, 4, 0, 0, 0, 5, 3, 2, 2, 0
\end{equation}

\begin{center}
\begin{tabular}{|p{1cm}|p{0.3cm}|p{0.3cm}|p{0.3cm}|p{0.3cm}|p{0.3cm}|p{0.3cm}|p{0.3cm}|p{0.3cm}|p{0.3cm}|p{0.3cm}|p{0.3cm}|p{0.3cm}|p{0.3cm}|p{0.3cm}|p{0.3cm}|p{0.3cm}|p{0.3cm}|p{0.3cm}|p{0.3cm}|p{0.3cm}|p{0.3cm}|p{0.3cm}|p{0.3cm}|p{0.3cm}|}
\hline
 \multicolumn{25}{|c|}{Table 1. The sequence  of $D_{n}(1,x)$ over $\mathbb{Z}_5$ for $0 \leq  n \leq 23$\label{B1}} \\
 \hline
$n$ & $0$ & $1$ & $2$ & $3$ & $4$ & $5$ & $6$ & $7$ & $8$ & $9$ & $10$ & $11$ & $12$ & $13$ & $14$ & $15$ & $16$ & $17$ & $18$ & $19$ & $20$ & $21$ & $22$ & $23$\\
\hline
$x = 0$ & $2$ & $1$ & $1$ & $1$ &$1$ &$1$ &$1$ &$1$ &$1$ &$1$ &$1$ &$1$ &$1$ &$1$ &$1$ &$1$ &$1$ &$1$ &$1$ &$1$ &$1$ &$1$ &$1$ &$1$ \\
\hline
$x = 1$ & $2$ & $1$ & $4$ & $3$ &$4$ &$1$ &$2$ &$1$ &$4$ &$3$ &$4$ &$1$ &$2$ &$1$ &$4$ &$3$ &$4$ &$1$ &$2$ &$1$ &$4$ &$3$ &$4$ &$1$ \\
\hline
$x = 2$ & $2$ & $1$ & $2$ & $0$ &$1$ &$1$ &$4$ &$2$ &$4$ &$0$ &$2$ &$2$ &$3$ &$4$ &$3$ &$0$ &$4$ &$4$ &$1$ &$3$ &$1$ &$0$ &$3$ &$3$ \\
\hline
$x = 3$ & $2$ & $1$ & $0$ & $2$ &$2$ &$1$ &$0$ &$2$ &$2$ &$1$ &$0$ &$2$ &$2$ &$1$ &$0$ &$2$ &$2$ &$1$ &$0$ &$2$ &$2$ &$1$ &$0$ &$2$ \\
\hline
$x = 4$ & $2$ & $1$ & $3$ & $4$ &$2$ &$1$ &$3$ &$4$ &$2$ &$1$ &$3$ &$4$ &$2$ &$1$ &$3$ &$4$ &$2$ &$1$ &$3$ &$4$ &$2$ &$1$ &$3$ &$4$ \\
\hline
\end{tabular}
\end{center}

\vskip 0.3in 

\begin{center}
\begin{tabular}{|p{1cm}|p{0.3cm}|p{0.3cm}|p{0.3cm}|p{0.3cm}|p{0.3cm}|p{0.3cm}|p{0.3cm}|p{0.3cm}|p{0.3cm}|p{0.3cm}|p{0.3cm}|p{0.3cm}|p{0.3cm}|p{0.3cm}|p{0.3cm}|p{0.3cm}|p{0.3cm}|p{0.3cm}|p{0.3cm}|p{0.3cm}|p{0.3cm}|p{0.3cm}|p{0.3cm}|p{0.3cm}|}
\hline
 \multicolumn{25}{|c|}{Table 2. The sequence of $D_{n}(1,x)$ over $\mathbb{Z}_7$ for $0 \leq n \leq 47$\label{B2}} \\
 \hline
$n$ & $0$ & $1$ & $2$ & $3$ & $4$ & $5$ & $6$ & $7$ & $8$ & $9$ & $10$ & $11$ & $12$ & $13$ & $14$ & $15$ & $16$ & $17$ & $18$ & $19$ & $20$ & $21$ & $22$ & $23$\\
\hline
$x = 0$ & $2$ & $1$ & $1$ & $1$ &$1$ &$1$ &$1$ &$1$ &$1$ &$1$ &$1$ &$1$ &$1$ &$1$ &$1$ &$1$ &$1$ &$1$ &$1$ &$1$ &$1$ &$1$ &$1$ &$1$ \\
\hline
$x = 1$ & $2$ & $1$ & $6$ & $5$ &$6$ &$1$ &$2$ &$1$ &$6$ &$5$ &$6$ &$1$ &$2$ &$1$ &$6$ &$5$ &$6$ &$1$ &$2$ &$1$ &$6$ &$5$ &$6$ &$1$ \\
\hline
$x = 2$ & $2$ & $1$ & $4$ & $2$ &$1$ &$4$ &$2$ &$1$ &$4$ &$2$ &$1$ &$4$ &$2$ &$1$ &$4$ &$2$ &$1$ &$4$ &$2$ &$1$ &$4$ &$2$ &$1$ &$4$ \\
\hline
$x = 3$ & $2$ & $1$ & $2$ & $6$ &$0$ &$3$ &$3$ &$1$ &$6$ &$3$ &$6$ &$4$ &$0$ &$2$ &$2$ &$3$ &$4$ &$2$ &$4$ &$5$ &$0$ &$6$ &$6$ &$2$ \\
\hline
$x = 4$ & $2$ & $1$ & $0$ & $3$ &$3$ &$5$ &$0$ &$1$ &$1$ &$4$ &$0$ &$5$ &$5$ &$6$ &$0$ &$4$ &$4$ &$2$ &$0$ &$6$ &$6$ &$3$ &$0$ &$2$ \\
\hline
$x = 5$ & $2$ & $1$ & $5$ & $0$ &$3$ &$3$ &$2$ &$1$ &$5$ &$0$ &$3$ &$3$ &$2$ &$1$ &$5$ &$0$ &$3$ &$3$ &$2$ &$1$ &$5$ &$0$ &$3$ &$3$ \\
\hline
$x = 6$ & $2$ & $1$ & $3$ & $4$ &$0$ &$4$ &$4$ &$1$ &$5$ &$6$ &$4$ &$3$ &$0$ &$3$ &$3$ &$6$ &$2$ &$1$ &$3$ &$4$ &$0$ &$4$ &$4$ &$1$ \\
\hline
\end{tabular}
\end{center}

\vskip 0.3in

\begin{center}
\begin{tabular}{|p{1cm}|p{0.3cm}|p{0.3cm}|p{0.3cm}|p{0.3cm}|p{0.3cm}|p{0.3cm}|p{0.3cm}|p{0.3cm}|p{0.3cm}|p{0.3cm}|p{0.3cm}|p{0.3cm}|p{0.3cm}|p{0.3cm}|p{0.3cm}|p{0.3cm}|p{0.3cm}|p{0.3cm}|p{0.3cm}|p{0.3cm}|p{0.3cm}|p{0.3cm}|p{0.3cm}|p{0.3cm}|}
\hline
 \multicolumn{25}{|c|}{Table 2 (contd.) The sequence of $D_{n}(1,x)$ over $\mathbb{Z}_7$ for $0 \leq n \leq 47$ (contd.) \label{B2}} \\
 \hline
$n$ & $24$ & $25$ & $26$ & $27$ & $28$ & $29$ & $30$ & $31$ & $32$ & $33$ & $34$ & $35$ & $36$ & $37$ & $38$ & $39$ & $40$ & $41$ & $42$ & $43$ & $44$ & $45$ & $46$ & $47$\\
\hline
$x = 0$ & $1$ & $1$ & $1$ & $1$ &$1$ &$1$ &$1$ &$1$ &$1$ &$1$ &$1$ &$1$ &$1$ &$1$ &$1$ &$1$ &$1$ &$1$ &$1$ &$1$ &$1$ &$1$ &$1$ &$1$ \\
\hline
$x = 1$ & $2$ & $1$ & $6$ & $5$ &$6$ &$1$ &$2$ &$1$ &$6$ &$5$ &$6$ &$1$ &$2$ &$1$ &$6$ &$5$ &$6$ &$1$ &$2$ &$1$ &$6$ &$5$ &$6$ &$1$ \\
\hline
$x = 2$ & $2$ & $1$ & $4$ & $2$ &$1$ &$4$ &$2$ &$1$ &$4$ &$2$ &$1$ &$4$ &$2$ &$1$ &$4$ &$2$ &$1$ &$4$ &$2$ &$1$ &$4$ &$2$ &$1$ &$4$ \\
\hline
$x = 3$ & $5$ & $6$ & $5$ & $1$ &$0$ &$4$ &$4$ &$6$ &$1$ &$4$ &$1$ &$3$ &$0$ &$5$ &$5$ &$4$ &$3$ &$5$ &$3$ &$2$ &$0$ &$1$ &$1$ &$5$ \\
\hline
$x = 4$ & $2$ & $1$ & $0$ & $3$ &$3$ &$5$ &$0$ &$1$ &$1$ &$4$ &$0$ &$5$ &$5$ &$6$ &$0$ &$4$ &$4$ &$2$ &$0$ &$6$ &$6$ &$3$ &$0$ &$2$ \\
\hline
$x = 5$ & $2$ & $1$ & $5$ & $0$ &$3$ &$3$ &$2$ &$1$ &$5$ &$0$ &$3$ &$3$ &$2$ &$1$ &$5$ &$0$ &$3$ &$3$ &$2$ &$1$ &$5$ &$0$ &$3$ &$3$ \\
\hline
$x = 6$ & $5$ &$6$ &$4$ &$3$ &$0$ &$3$ &$3$ &$6$ &$2$ &$1$ &$3$ &$4$ &$0$ &$4$ &$4$ &$1$ &$5$ &$6$ &$4$ &$3$ &$0$ &$3$ &$3$ &$6$ \\
\hline
\end{tabular}
\end{center}


\section{Python codes}\label{AppC}

In Appendix~\ref{AppC}, we present the python codes to generate permutation polynomials and complete permutation polynomials from reversed Dickson polynomials over $\mathbb{F}_p$ and to compute cycle types of reversed Dickson permutation polynomials. 

\subsubsection*{Python code to generate permutation polynomials from RDPs}

\begin{verbatim}
import math
from sympy import symbols, div, Poly

def degree_mod(poly, mod):
    _, r = div(poly, mod)
    return r.degree()

x = symbols('x')
d0 = Poly(2, x)
d1 = Poly(1, x)
n0 = 1
m = x**p - x
p = 3

def mod_poly(poly, mod):
    coeffs = poly.all_coeffs()
    return Poly([coeff % mod for coeff in coeffs], x)

for n in range(2, p**2):
    for k in range(n0 + 1, n + 1):
        d = d1 - (x * d0)
        q, r = div(d, x**p - x)
        d = mod_poly(r, p)
        d0 = d1
        d1 = d
        n0 = n
        f = d + x
        d_is_pp = True
        for t in range(1, p - 1):
            q, r = div(d**t,m)
            modded_coeffs = [coeff % p for coeff in r.all_coeffs()]
            modded_poly = Poly(modded_coeffs, x)
            if modded_poly.degree() > p - 2:
                d_is_pp = False
                break
        q, r = div(d**(p-1),m)
        modded_coeffs = [coeff % p for coeff in r.all_coeffs()]
        modded_poly = Poly(modded_coeffs, x)
        if modded_poly.degree() != p - 1:
            d_is_pp = False
        if d_is_pp:
            print(k,':',d, " is a pp")
            for i in range(2, 200):
                substituted_poly = d.subs(x, i)
                r = substituted_poly % p 
                if r == i:
                    print(i, "is a fixed point, with n =", n)

\end{verbatim}

\subsubsection*{Python code to generate complete permutation polynomials from RDPs}

\begin{verbatim}
from sympy import symbols, div, Poly

# function to calculate degree of a polynomial modulo a given value
def degree_mod(poly, mod):
    _, r = div(poly, mod)
    return r.degree()

x = symbols('x')
d0 = Poly(2, x)
d1 = Poly(1, x)
n0 = 1
p = 5
m = x**p - x

def mod_poly(poly, mod):
    coeffs = poly.all_coeffs()
    return Poly([coeff % mod for coeff in coeffs], x)

for n in range(2, p**2):
    for k in range(n0 + 1, n + 1):
        d = d1 - (x * d0)
        q, r = div(d, x**p - x)
        d = mod_poly(r, p)
        d0 = d1
        d1 = d

        n0 = n
        f = d + x
    
    # Hermite's Criterion for d
        d_is_pp = True
        for t in range(1, p - 1):
            q, r = div(d**t,m)
            modded_coeffs = [coeff % p for coeff in r.all_coeffs()]
            modded_poly = Poly(modded_coeffs, x)
            if modded_poly.degree() > p - 2:
                d_is_pp = False
                break  # If d is not a PP, break the t loop
        q, r = div(d**(p-1),m)
        modded_coeffs = [coeff % p for coeff in r.all_coeffs()]
        modded_poly = Poly(modded_coeffs, x)
        if modded_poly.degree() != p - 1:
            d_is_pp = False
    # If d passes Hermite's Criterion, check 
        if d_is_pp:
            #print("d is a PP for n = ",n)
        # Hermite's Criterion for f
            f_is_pp = True
            for t in range(1, p - 1):
                q, r = div(f**t,m)
                modded_coeffs = [coeff % p for coeff in r.all_coeffs()]
                modded_poly = Poly(modded_coeffs, x)
                if modded_poly.degree() > p - 2:
                    f_is_pp = False
                    break
            q, r = div(f**(p-1),m)
            modded_coeffs = [coeff % p for coeff in r.all_coeffs()]
            modded_poly = Poly(modded_coeffs, x)
            if modded_poly.degree() != p - 1:
                f_is_pp = False
    
            if d_is_pp & f_is_pp:
                print('d = ',{d})
                print("d is a CPP for n =", n)

\end{verbatim}

\subsubsection*{Python code to compute cycle types of reversed Dickson permutation polynomials}

\begin{verbatim}

from sympy import symbols, expand

# Define the polynomial
x = symbols('x')
poly = expand( 1 + x)

# Define the field size
field_size = 3

# Compute the permutation
permutation = [(i, poly.subs(x, i) % field_size) for i in range(field_size)]

# Function to convert permutation to cycle notation
def permutation_to_cycles(permutation):
    # Start with no cycles
    cycles = []
    # While there is an unprocessed part of the permutation
    while permutation:
        # Start a new cycle
        n, p = permutation.pop()
        cycle = [n]
        # While the cycle is not closed
        while p != cycle[0]:
            # Add to the cycle
            for i, (n2, p2) in enumerate(permutation):
                if n2 == p:
                    cycle.append(n2)
                    p = p2
                    permutation.pop(i)
                    break
        # Add the cycle to the list of cycles
        cycles.append(cycle)
    # Return the cycles
    return cycles

# Compute the cycles
cycles = permutation_to_cycles(permutation)

# Print the cycles
for cycle in cycles:
    print("(", " ".join(str(n) for n in cycle), ")")
    
\end{verbatim}

\subsubsection*{Python code to compute fixed points}

\begin{verbatim}

from sympy import symbols, div, Poly

# function to calculate degree of a polynomial modulo a given value
def degree_mod(poly, mod):
    _, r = div(poly, mod)
    return r.degree()

p = (enter the p value)
a = []
x = symbols('x')
d0 = Poly(2, x)
d1 = Poly(1, x)
n0 = 1
m = x**p - x


def mod_poly(poly, mod):
    coeffs = poly.all_coeffs()
    return Poly([coeff % mod for coeff in coeffs], x)

for n in range(2, p**2):
    for k in range(n0 + 1, n + 1):
        d = d1 - (x * d0)
        q, r = div(d, x**p - x)
        d = mod_poly(r, p)
        d0 = d1
        d1 = d

        n0 = n
        f = d + x
    
        # Hermite's Criterion for d
        d_is_pp = True
        for t in range(1, p - 1):
            q, r = div(d**t,m)
            modded_coeffs = [coeff % p for coeff in r.all_coeffs()]
            modded_poly = Poly(modded_coeffs, x)
            if modded_poly.degree() > p - 2:
                d_is_pp = False
                break  # If d is not a PP, break the t loop
        q, r = div(d**(p-1),m)
        modded_coeffs = [coeff % p for coeff in r.all_coeffs()]
        modded_poly = Poly(modded_coeffs, x)
        if modded_poly.degree() != p - 1:
            d_is_pp = False
        # If d passes Hermite's Criterion, check 
        if d_is_pp:
            print(k, ':', d, " is a pp")
            for i in range(2, 200):
                substituted_poly = d.subs(x, i)
                r = substituted_poly % p 
                if r == i:
                    a.append(r)
                    fixed_pt = len(a)
                    print(i, "is a fixed point, with n =", n)
                    
fixed_pt_count = len(a)
print('The total number of fixed points:', fixed_pt_count)
    
\end{verbatim}


\begin{thebibliography}{99}

\bibitem{DS}
A. Diene, M. A. Salim. \emph{Fixed Points of the Dickson Polynomials of the Second Kind}, J. Appl. Math. 2013 1 -- 7, 2013

\bibitem{Ahmad}
S. Ahmad, \emph{Cycle structure of automorphisms of finite cyclic groups}, Journal of Combinatorial Theory A, vol. 6, pp. 370 -- 374, 1969.

\bibitem{BZ}
L. A. Bassalygo, V. A. Zinoviev, \emph{Permutation and complete permutation polynomials}, Finite Fields Their Appl., \textbf{33} (2008), 198 -- 211.

\bibitem{CMT}
A, \c{C}e\c{s}melio\u{g}lu, W. Meidl, A. Topuzo\u{g}lu, \emph{On the cycle structure of permutation polynomials}, Finite Fields Their Appl., \textbf{14} (2008), 593 -- 614.

\bibitem{CC}
M. Cipu, S. D. Cohen, \emph{Dickson polynomial permutations}, Finite Fields and Applications, Contemporary Mathematics, vol. 461, 2008, 79 -- 91.

\bibitem{Dobbertin}
H. Dobbertin, Almost perfect nonlinear power functions on $GF(2^n)$: A new case for $n$ divisible by $5$, in: Finite Fields and Applications, Springer, Berlin, 2001, pp. 113 -- 121.

\bibitem{NF}
N. Fernando, 
\emph{Reversed Dickson polynomials of the $(k+1)$-th kind over finite fields}, J. Number Theory {\bf 172} (2017), 234 -- 255. 

\bibitem{NF1}
N. Fernando, 
\emph{A note on permutation binomials and trinomials over finite fields}, New Zealand J. Math. {\bf 48} (2018), 25 -- 29.. 

\bibitem{NF2}
N. Fernando, 
\emph{Reversed Dickson polynomials of the $(k+1)$-th kind over finite fields, II}, To appear in Contributions to Discrete Mathematics. 

\bibitem{GHM}
D. G\"{o}rcs\"{o}s, G. Horv\'{a}th, A. M\'{e}sz\'{a}ros,
\emph{Permutation polynomials over finite rings}, Finite Fields Appl., 49 (2018), pp. 198 -- 211.

\bibitem{HQZ}
S. Hong, X. Qin, W. Zhao, 
\emph{Necessary conditions for reversed Dickson polynomials of the second kind to be permutational}, 
Finite Fields Their Appl., \textbf{37}(2016), 54 -- 71.

\bibitem{HL}
X-D. Hou, T. Ly, 
\emph{Necessary conditions for reversed Dickson polynomials to be permutational}, 
Finite Fields Their Appl., \textbf{16}(6) (2010), 436 -- 448.

\bibitem{HMSY}
X-D. Hou, G. L. Mullen, J. A. Sellers, J. L. Yucas, 
\emph{Reversed Dickson polynomials over finite fields}, 
Finite Fields Their Appl., \textbf{15}(6) (2009), 748 -- 773.

\bibitem{DM}
R. Lidl, G.L. Mullen, G. Turnwald, \emph{Dickson Polynomials}, Longman Scientific and Technical, Essex,
United Kingdom, 1993.

\bibitem{LM}
R. Lidl and G. L. Mullen, \emph{Cycle structure of Dickson permutation polynomials}, Mathematical Journal of Okayama University, vol. 33, pp. 1 -- 11, 1991.

\bibitem{MN87}
G. L. Mullen, H. Niederreiter, \emph{Dickson polynomials over finite fields and complete mappings}, Canad Math Bull, 1987, 1: 19 -- 27

\bibitem{MP14}
A. Muratovi\'{c}-Ribi\'{c}, E. Pasalic, \emph{A note on complete polynomials over finite fields and their applications in cryptography}, Finite Fields Appl, 2014, 25: 306 -- 315

\bibitem{NR82}
H. Niederreiter, K.H. Robinson, \emph{Complete mappings of finite fields}, J. Aust. Math. Soc. A 33 (2) (1982) 197 -- 212.

\bibitem{NB68}
W. N\"{o}bauer, \emph{Uber eine Klasse von Permutationspolynomen und die dadurch dargestellten Gruppen}, J.
Reine Angew. Math., 231 (1968) 215 -- 219.

\bibitem{Rivest}
R. L. Rivest, \emph{Permutation polynomials modulo $2^w$}, Finite Fields Appl., 7 (2001), pp. 287 -- 292.

\bibitem{Shur23}
I. Schur, \"{U}ber den Zusammenhang zwischen einemem Problem der Zahlentheorie und einem Satz \"{u}iber
algebraische Funktionen, Sitzungsber. Akad. Wiss. Berlin (1923) 123 -- 134.

\bibitem{ST}
J. Sun, O.Y. Takeshita, \emph{Interleavers for turbo codes using permutation polynomials over integer rings}, IEEE Trans. Inf. Theory 51 (1) (2005) 101 -- 119.

\bibitem{T}
O.Y. Takeshita, \emph{On maximum contention-free interleavers and permutation polynomials over integer rings}, IEEE Trans. Inf. Theory 52 (3) (2006) 1249 -- 1253.

\bibitem{QD}
L. Qu, C. Ding, 
\emph{Dickson polynomials of the second kind that permute $\mathbb{Z}_m$}, SIAM J. Discrete Math. 28, No. 2, 722 -- 735 (2014).

\bibitem{WY}
Q. Wang, J. L. Yucas, \emph{Dickson polynomials over finite fields}, Finite Fields Appl. 18 (2012) 814 -- 831.

\end{thebibliography}
\end{document}